\documentclass[a4paper,12pt,oneside,reqno]{amsart}
\usepackage{hyperref}
\usepackage[headinclude,DIV13]{typearea}
\areaset{15.1cm}{25.0cm}
\usepackage{amssymb,amsmath,amsthm,bbm}

\usepackage{txfonts}
\usepackage[latin1] {inputenc}
\usepackage{graphicx, psfrag}
\usepackage{mathrsfs}
\usepackage{enumerate}

\numberwithin{equation}{section}

\newtheorem{theorem}{Theorem}[section]
\newtheorem{lemma}[theorem]{Lemma}
\newtheorem{proposition}[theorem]{Proposition}

\newtheorem{remark}[theorem]{Remark}

\theoremstyle{definition}
\newtheorem{definition}[theorem]{Definition}

\newcommand{\abs}[1]{\left| #1\right|}

\newcommand{\eee}{\mathrm e}

\newcommand{\calA}{\mathcal{A}}

\newcommand{\calD}{\mathcal{D}}

\newcommand{\calF}{\mathcal{F}}

\newcommand{\calH}{\mathcal{H}}

\newcommand{\calL}{\mathcal{L}}
\newcommand{\calM}{\mathcal{M}}
\newcommand{\calN}{\mathcal{N}}

\newcommand{\calP}{\mathcal{P}}

\newcommand{\fra}{\mathfrak{a}}

\newcommand{\bbE}{\mathbb{E}}

\newcommand{\bbL}{\mathbb{L}}
\newcommand{\bbM}{\mathbb{M}}
\newcommand{\bbN}{\mathbb{N}}

\newcommand{\bbP}{\mathbb{P}}
\newcommand{\bbQ}{\mathbb{Q}}
\newcommand{\bbR}{\mathbb{R}}

\newcommand{\bbT}{\mathbb{T}}

\newcommand{\sfc}{\mathsf c}

\newcommand{\sfq}{{\sf q}}

\newcommand{\sfC}{{\sf C}}
\newcommand{\sfD}{{\sf D}}

\newcommand{\utheta}{\underline{\theta}}

\newcommand{\Td}{\bbT^d}

\newcommand{\Tod}{\bbT^d}

\newcommand{\TodN}{\bbT_N^d}

\newcommand{\lb}{\left(}
\newcommand{\rb}{\right)}
\newcommand{\lbr}{\left\{}
\newcommand{\rbr}{\right\}}

\newcommand{\dd}{{\rm d}}

\newcommand{\step}[1]{S{\small TEP}\,#1.}

\def\d{\delta}

\def\f{\phi}

\def\th{\theta}

\def\eu{{1\kern-.25em\rm{I}}}
\def\f1{{1\kern-.25em\rm{I}}}
\def\R{{\mathbb R}}  
\def\N{{\mathbb N}}  

\def\la{\langle}
\def\ra{\rangle}

\def\AA{{\mathfrak A}}
\def\BB{{\mathfrak B}}

\def\1{\ifmmode {1\hskip -3pt \rm{I}}
\else {\hbox {$1\hskip -3pt \rm{I}$}}\fi} 

\newcommand{\smo}[1]{{\mathrm o}\lb #1\rb }

\newcommand{\df}{\equiv}   

\newcommand{\be}[1]{\begin{equation}\label{#1}}
\newcommand{\ee}{\end{equation}}

\begin{document}

\title[Hydrodynamic limit for local mean-field dynamics]{The hydrodynamic limit for local mean-field dynamics with unbounded spins}

\author[A. Bovier]{Anton Bovier}
 \address{A. Bovier\\Institut f\"ur Angewandte Mathematik\\
Rheinische Friedrich-Wilhelms-Universit\"at\\ Endenicher Allee 60\\ 53115 Bonn, Germany }
\email{bovier@uni-bonn.de}
\author[D. Ioffe]{Dmitry Ioffe}
 \address{D. Ioffe\\ William Davidson Faculty of Industrial Engineering and 
 Management\\Technion\\ Haifa 32000, Israel}
\email{ieioffe@technion.ac.il}

\author[P. Müller]{Patrick Müller}
 \address{P. Müller\\Boston Consulting Group, Im Mediapark 8, 50670 Köln, Germany}

\begin{abstract}
We consider the dynamics of   a class of spin systems with unbounded 
spins interacting with 
local mean-field interactions. We prove convergence of the empirical 
measure to the solution of a McKean-Vlasov equation in the 
hydrodynamic
limit and propagation of chaos. This extends earlier results of 
Gärtner, Comets and others for bounded spins or strict mean-field interactions. 
\end{abstract}

\thanks{A.B. is partially supported through the German Research Foundation in 
the Collaborative Research Center 1060 "The Mathematics of Emergent Effects", 
the Hausdorff Center for Mathematics (HCM), and  the 
Cluster of Excellence ``ImmunoSensation'' at Bonn University.}

\date{\today}

\maketitle

\section{Introduction and results}\label{1.1.1}

In this paper we consider coupled systems of $N\in \bbN$ stochastic differential equations (sde) of the form 
\be{eq:SDE-N}
\dd \theta^N_i (t) = -\psi^\prime \lb \theta^N_i (t )\rb \dd t+ 
\frac{1}{N^d}\sum_{j\in \TodN}
J\lb\frac{j-i}{N}\rb\theta^N_j (t )\dd t +\dd B_i (t)  ,  \quad i\in \TodN,
\ee
Here we denote by $\TodN\equiv \{1,\dots, N\}^d$ the $d$-dimensional 
discrete torus of side-length $N$. $\theta^N_i(t)$ take values in $\R$, 
$\psi:\R\rightarrow \R$ is a local potential that we will assume for simplicity
to 
be a   polynomial of degree
 $2k$
 \footnote{Arguments developed in Section~\ref{sec:MF} require smoothness and certain growth 
properties of $\psi^\prime$ at infinity and could be readily extended to a larger class of local 
potentials.}, with $k\geq 2$, that is
\be{eq:psi-con1}
\psi (\theta ) = \theta^{2k} + \text{lower order terms}. 
\ee
The interaction $J:\Tod\rightarrow\R_+$ will be assumed to be a smooth 
symmetric function on the $d$-dimensional  unit torus $\Tod$.
Finally, $B_i, i
\in \N$ are iid Brownian motions.

We are interested in describing the behaviour of this system in the 
limit as $N\uparrow\infty$. To to so, we consider the \emph{empirical
process},
\be{empdef.1}
\mu^N: \R_+\rightarrow \bbM_1(\Tod\times \R),
\ee
given by 
\be{empdef.11}
\mu^N_t=\frac 1{N^d}\sum_{k\in \TodN}\d_{\left(k/N, \theta_k^N(t)\right)},
\quad t\in \R_+. 
\ee
In terms of the empirical process, the equations \eqref{eq:SDE-N} can be written as
\be{empdef.2}
\dd \theta_i(t)^N= -\psi^\prime \lb \theta_i (t )\rb \dd t
+ \int_{\Tod}\int_{\R} J(i/N-y) \theta  \mu^N_t(\dd y, \dd \theta)\dd t
+\dd B_i (t)  ,  \quad i\in \TodN,
\ee
Now, if $\mu^N$ converges to some measure $\mu$, then 
{ it is reasonable to expect that} 
in the limit $N\uparrow\infty$, 
the $\theta_i$ will be independent diffusions and that 
their empirical distribution 
{ should}, by the law of large numbers, converge to a measure
 ${ \mu_t (d x, d\th )} = \rho_t(x,d\th)\dd x$, where, 
 {for any $x\in\Tod$, }
 $\rho_t(x,d\th)$ 
 is the law of the 
 diffusion 
 \be{empdef.3}
\dd \theta(t)= -\psi^\prime \lb \theta (t )\rb \dd t
+ \int_{\Tod}\int_{\R} J(x-y) \theta  \mu_t(\dd y, \dd \theta)\dd t
+\dd B (t)  .
\ee
This self-consistent equation is called the
 \emph{McKean-Vlasov equation}.
The models we consider, and in fact an even richer class of models including random interactions and 
potentials, was studied from the point of view of large deviations by one of us \cite{Patrick} where also 
an 
extensive review of  the history of these models is given. The main purpose of the present paper is to 
give a simple and transparent  proof of just the law of large numbers (or hydrodynamic limit).  Earlier 
and similar result for more restricted classes of models with strict mean-field interaction 
(i.e. $J$ constant) goes back to Gärtner \cite{Gaertner88} and Comets and Eisele
\cite{ComEis88}, {see also lecture notes \cite{sznitman91} for a comprehensive account}.  A somewhat 
non-rigorous derivation in the local mean-field case  
with  bounded spins
was given in Katsoulakis et al. 
\cite{Greeks05}.

The convergence proof we present here, 
under assumption of sufficiently regular initial distributions - see Theorem \ref{thm:HDL} below, 
  has two  main steps. First, one shows that
the  associated local  mean-field McKean-Vlasov system \eqref{eq:L-MF-MV}, 
as specified in the  Definition~\ref{mkprob.1} below, 
has a unique solution with good regularity properties. In fact we will show that the measure 
 $\rho_t(x,d\th)$
is absolutely continuous with respect to Lebesgue measure and has 
a smooth density that is the solution of a certain partial differential 
equation. This will be done in Section \ref{sec:MF} using a {fixed point} argument. The remainder of the proof
relies on existence, unicity and regularity results for the local mean-field system 
of equations \eqref{eq:L-MF-MV}.  This will be done by a relative entropy estimate. 
In Section \ref{sec:PC} we prove an  additional propagation of chaos result 
that  is also based on 
appropriate relative entropy estimates, 
which in their turn rely on  Girsanov transforms and regularity results for solutions of 
\eqref{eq:L-MF-MV}. 
In the concluding Section \ref{sec:LD} we outline a proof of a large deviation principle for 
empirical measures $\mu^N$.

\section{Local McKean-Vlasov   equation}
\label{sec:MF} 
In the sequel we say that a function $f$ is {\em smooth} on the  closure 
$\bar D$ of an open 
domain $D$ 
if  $f$  is  $\sfC^\infty$ on $D$ 
with derivatives of all order  having continuous extensions to 
 $\bar D$.

\subsection{Heat kernels for $1$-dimensional diffusions.}
To set up the McKean-Vlasov system in a rigorous way, we consider, for 
smooth   functions $h:\R_+\rightarrow \R$ the sde 
\be{eq:theta-t} 
\dd\theta(t) = \lb h(t) - 
{\psi}^{\prime} (\theta(t) )\rb \dd t +\dd B(t). 
\ee
The solution of this equation is a time-inhomogeneous 
Markov process whose generator is
the closure of the operator 
\be{eq:L-h}
L_{h(t)} = 
\frac{1}{2}
\eee^{{2}\psi (\theta )} \partial_{\theta}\lb 
\eee^{ - {2}\psi (\theta )}\partial_\theta \rb + h(t) \partial_\theta = L_0 + h(t) \partial_\theta .
\ee
It is useful to consider $L_{h}$ as an operator on the Hilbert space 
$L^2\left(\R, \eee^{-{2}\psi}\right)$, since $L_0$ is a self-adjoint operator 
on this space. Below, $\langle\cdot , \cdot\rangle_\psi$, 
and $\|\cdot\|_{2,\psi}$ denote the scalar 
product and the norm on $L^2 \lb \bbR , \eee^{-{2}\psi }\rb$, respectively.

The formal adjoint of $L_{h}$ on $L^2 \lb \bbR , \eee^{-{2}\psi }\rb$ acts on 
functions $\rho:\R_+\times\R$ as 
\be{eq:L-h-star}
 \left(L_{h}^*\rho\right)_t(\theta)  = \left(L_0 \rho\right)_t(\theta)  - h(t)  
 \eee^{{2}\psi(\theta) } \partial_\theta\lb \eee^{-{2}\psi(\theta) }\rho_t(\theta)\rb . 
\ee
Condition \eqref{eq:psi-con1} implies that 
$L_0$ 
has
 compact resolvent on $L^2 \lb \bbR , \eee^{-{2}\psi }\rb$. It has a smooth transition 
density (with respect to $\eee^{-{2}\psi}$)
\be{eq:q-psi}
\sfq_t^0 (\eta , \theta   ) = \sum_{1}^{\infty} \eee^{ - \lambda_i t}\phi_i (\eta  )\phi_i (\theta  ) , 
\ee
where $\lbr \phi_i\rbr_{i\in \N}$ is a complete orthonormal basis of eigenfunctions of $L_0$ and 
$\lambda_i$ are the corresponding eigenvalues.  If 
\be{blub.1}
 \rho_0\in  \calD (L_0 ) = \lbr \sum_i a_i\phi_i \quad{\rm with}\quad \sum_i \lambda_i^2 a_i^2 <\infty\rbr 
\ee
is an initial density, then $\rho_t (\theta  ) = 
\langle \rho_0 , \sfq_t^{{0}} (\cdot, \theta  )\rangle_\psi $ 
is the density at time 
$t$, and it solves the Fokker-Planck equation
 $\partial_t \rho_t = L_0 \rho_t$ with 
initial condition $\rho_0$. 

We first show that the law of the solution of the sde \eqref{eq:theta-t}
is absolutely continuous with a density that is the unique strong solution 
of the Fokker-Planck equation associated to the operator $L_{h(t)}^*$. 
Namely, there exists  a  $\sfC^\infty \lb (0, \infty )\times\bbR^2 , \bbR\rb$ 
kernel (see \eqref{eq:qth} below) 
$(t, \theta , \eta )\mapsto \sfq_t^h (\theta , \eta )$, 
such that the following holds:

\begin{lemma} 
\label{lem:prop-dens} 
Let  the initial distribution of   the diffusion 
\eqref{eq:theta-t} be absolutely continuous with respect to the 
measure $\eee^{-{2}\psi(\theta)}\dd\theta$ with density $\rho_0$. 
Assume that $h$ is smooth on $\bbR_+$. 
Then,  
for any $t >0$, 
the distribution of 
$\theta(t)$ at time $t$ is absolutely continuous with respect to the 
measure $\eee^{-{2}\psi(\theta)}\dd\theta$ with density $\rho^h_t$,
where, 
\be{eq:rhoth} 
\rho_t^h (\eta ) = \int_{\bbR} \rho_0 (\theta ) \sfq^h_t (\theta , \eta )\dd\theta 
\ee
is the classical solution of the Fokker-Planck equation 
\be{eq:FP-Lh} 
\partial_t \rho_t^h  = L_{h(t)}^* \rho_t^h . 
\ee
\end{lemma}

\begin{proof} The proof of this lemma is adapted from an argument by 
Rogers \cite{Rog84}. It is based on an application of Girsanov's formula.
Define 
\be{rogers.1}
X(t) \equiv -\int_0^t \lb \psi^\prime ( B(s)) - h(s)\rb \dd B(s).
\ee 
Let $\langle X\rangle_t$ be the quadratic variation process of $X (t )$. 
The results in \cite{MU12,KR15} imply that 
\be{eq:Xt} 
 \eee^{X(t) - \frac{1}{2}\langle X\rangle_t}
\ee
is a martingale.  Moreover, Girsanov's formula holds for $\theta(t)$ in \eqref{eq:theta-t}. Namely, 
for any bounded and continuous $f$ on $\bbR$, 
\be{eq:Girsanov} 
\bbE_\theta^h\left[  f (\theta(t) )\right] = \bbE_\theta^{{\rm BM}}\left[ \eee^{X(t) - \frac{1}{2}\langle X\rangle_t} f (B(t) )\right], 
\ee
where  $\bbP^{{\rm BM}}$ is the law of the Brownian motion starting at $\theta$, and $\bbP_\theta^h$ is the 
law of the time-inhomogeneous diffusion \eqref{eq:theta-t}. 

Set $\tau_R = \min\lbr t~:~ \abs{B(t)}\geq R\rbr$.  Then \eqref{eq:Girsanov} implies: 
\be{eq:Girs-R} 
\bbE_\theta^{h}\left[ f (\theta(t) ) \right]= \lim_{R\to\infty} 
\bbE_\theta^{{\rm BM}}\left[ \eee^{X(t) - \frac{1}{2}\langle X\rangle_t} f (B(t) )\1_{\lbr \tau_R >t \rbr}\right] .
\ee
Let $\psi_R$ be a sequence of smooth functions such that
\be{girs.1}
\text{$\psi_R (\theta ) = \psi (\theta )$ if $\abs{\theta} \leq R$, and $\psi_R = const$ on both $(-\infty , -2R]$ and $[2R, \infty )$. } 
\ee
Note that  on $\lbr \tau_R >t \rbr$ 
\be{girs.2}
X(t)- \frac{1}{2}\langle X\rangle_t = -\int_0^t \lb \psi^\prime_R  ( B(s)) - h(s)\rb \dd B(s) - 
\int_0^t 
\frac{1}{2} \lb \psi^\prime_R  ( B(s) ) - h(s)\rb^2\dd s .
\ee
By Ito's formula, 
\be{girs.3}
-\int_0^t  \psi^\prime_R  ( B(s) ) \dd B(s) 
= \psi_R (\theta ) - \psi_R (B(t) ) +\frac{1}{2} \int_0^t \psi^{\prime\prime}_R (B(s))\dd s .
\ee
By partial integration,
\be{girs.4}
\int_0^t  h(s) \dd B(s) = h(t) B(t) -h (0) \theta - \int_0^t h^{\prime}(s) B(s) \dd s .
\ee
In view of \eqref{eq:Girs-R} we conclude that
\be{eq:Girs-Rog} 
\bbE_\theta^{h}\left[ f (\theta(t) )\right]  =  \eee^{\psi (\theta ) - h(0)\theta}\, 
\bbE_\theta^{{\rm BM}} \left[\eee^{\int_0^{{t}} F  (B(s) , h(s) , h^{\prime}(s) )\dd s} \, 
\eee^{h(t) B(t)  - \psi ( B(t) )} \, 
f (B(t) )\right]  , 
\ee
where 
\be{eq:f-bhhp} 
F (B , h , h^\prime ) 
\df \frac{1}{2} \psi^{\prime\prime} (B) - h^\prime B -\frac{1}{2}\lb \psi^\prime (B) - h\rb^2 .
\ee
By our assumptions on $\psi$ and $h$ for any $t$ there exists a finite 
constant $C (t, h )$ such that 
\be{eq:f-bound} 
\max_{s\leq t, B}  F(B , h(s) , h^{\prime}(s)) \leq C(t, h ) .
\ee
Proceeding as in \cite{Rog84} we infer that 
\be{eq:Heat-k} 
\bbE_\theta^{h}\left[ f (\theta(t) )\right]  = \eee^{\psi (\theta ) }
\int_{\bbR} {\mathsf q} _t^h (\theta , \eta )\eee^{-{2}\psi (\eta )} f (\eta )\dd\eta 
\ee
where the heat kernel $\sfq_t^h$ is given by
\be{eq:qth} 
\sfq_t^h (\theta , \eta ) = \gamma_t (\theta , \eta ) \phi_t^h (\theta  , \eta ) , 
\ee
with $\gamma_t (\theta , \eta ) = \frac{1}{\sqrt{2\pi t}} \eee^{- \lb \eta - \theta \rb^2/2t}$, and 
\be{eq:phith} 
\phi_t^h (\theta , \eta ) = \bbE_{\theta , \eta}^{{\rm BB}} \left[ \eee^{ h(t)\eta -h (0)\theta + 
\int_0^{{t}} F  (B_s , h(s) , h^{\prime} (s) )\dd s}\right]  
\ee
is an exponential functional of a Brownian bridge from $\theta$ to $\eta$ in time $t$. 
Then \eqref{eq:rhoth} follows. 
Next, 
 recall  that the Brownian bridge $B$ has  the convenient representation 
\be{eq:BB-BM} 
B_s = \theta +\frac{s}{t} (\eta - \theta ) + \lb W_s - \frac{s}{t}W_t\rb, \quad s\in [0,t],
\ee
 in terms of a Brownian motion $W$. 
Using this representation we can rewrite \eqref{eq:phith} as 
\be{eq:phith-BM} 
\begin{split}
\phi_t^h (\theta , \eta ) & = \bbE^{{\rm BM}} \left[ \eee^{ h(t)\eta -h (0) \theta + 
\int_0^{{t}} F  \lb \theta +\frac{s}{t} (\eta - \theta ) + \lb W_s - \frac{s}{t}W_t\rb, 
h(s), h^{\prime} (s )\rb \dd s}\right]  \\ 
& = 
\bbE^{{\rm BM}} \left[ \eee^{ h(t)\eta -h (0) \theta + t
\int_0^{1} F  \lb \theta +u  (\eta - \theta ) + \sqrt{t} \lb W_u - W_1\rb, 
h(ut ) , h^{\prime} (u t) \rb \dd u}\right]
\end{split}
\ee
Again, 
proceeding 
as in \cite{Rog84}, 
dominated convergence arguments imply that 
$\phi_t^h (\theta , \eta )$ is continuously differentiable in $t$ on 
$(0, \infty )$ 
and, for every $t>0$ it is  $\sfC^\infty$ in 
$\theta$ and $\eta$ on  $\bbR\times \bbR$. 
From this the claim \eqref{eq:FP-Lh}   of the lemma follows by a 
modification of standard computations employed on pp.160-161 
of \cite{Rog84} . 
\end{proof}

The proof of the preceding lemma readily yields a bound on the growth of 
$\rho_t$. 
\begin{lemma} 
\label{prop:L2}
Under the  conditions of Lemma~\ref{lem:prop-dens} 
assume in addition 
that  the the initial  density $\rho_0$ 
is bounded 
in $L^2 \lb \bbR , \eee^{-{2}\psi }\rb $, that is  $\|\rho_0\|_{2,\psi} <\infty$ .
Then, 
\be{eq:L2-rhoht} 
\|\rho_t^h\|_{2,\psi}^2 \leq \eee^{\frac{1}{2} \int_0^t {h}(s)^2\dd s } \| \rho_0\|_{2,\psi}^2 , 
\ee 
\end{lemma}
\begin{proof} 
 Note that \eqref{eq:f-bound},  \eqref{eq:qth} and \eqref{eq:phith} imply the 
 following bound: 
\be{tttt.1}
\sup_\eta \| \sfq_t^h (\cdot  , \eta )\|_\infty + 
\sup_\eta \| \partial_\eta \sfq_t^h (\cdot  , \eta )\|_\infty
< \infty 
\ee 
for any $t >0$. 
Hence, if $\rho_0$ is  compactly supported, \eqref{eq:rhoth} implies that, 
\be{tttt.2}
\|\rho_t^h\|_{2,\psi} + \|\partial_\eta \rho_t^h\|_{2,\psi} <\infty , 
\ee
for any $t >0$. 
On the other hand, using \eqref{eq:FP-Lh}, 
\begin {eqnarray}\label{tttt.3}\nonumber
\frac{\dd}{\dd t} \|\rho_t^h\|_{2,\psi}^2 &=& -2\|\partial_\eta 
\rho_t^h\|_{2,\psi}^2 +2 h(t) \int_\bbR 
\rho_t^h (\eta ) \partial_\eta \rho_t^h (\eta ) \eee^{-{2}\psi (\eta )}\dd\eta\\
&\leq&  \frac 12h(t)^2 \|\rho_t^h\|_{2,\psi}^2,
\end{eqnarray}
where the inequality  follows from the elementary fact that 
 $ab\leq \frac{a^2}{4} +b^2$, applied with $a=h(t)\rho^h_t(\eta)$ and 
 $b=\partial_\eta \rho_t^h (\eta )$. 
Integrating this differential inequality yields
\eqref{eq:L2-rhoht}. 
The general case follows by monotone convergence arguments. 
\end{proof}


A further consequence is a Lipshitz bound on the dependence of
the densities on the drift. 
Consider an $L^2 (\bbR , \eee^{-{2}\psi })$ initial density $\rho_0$ and 
let $h(t)$ and $g(t)$ be two 
time dependent smooth drifts on $\bbR_+$. 
\begin{lemma} 
\label{prop:stability}
{Define $D_t (\eta ) = \rho_t^g (\eta ) - \rho_t^h (\eta )$. } Then 
the $L^2 (\bbR , \eee^{-{2}\psi })$-norm of $D_t$ satisfies the following 
upper bound: 
\be{eq:Dt-bound} 
\left\|\rho_t^g (\eta ) - \rho_t^h (\eta )\right\|_{2,\psi}^2  \leq 
 \eee^{ \int_0^t \lb g(s)^2 +h(s)^2\rb \dd s } \| 
 \rho_0\|_{2,\psi}^2 \int_0^t \lb g(s) - h(s) \rb^2 \dd s .
\ee
\end{lemma}
\begin{proof} 
As in the proof of 
 Proposition~\ref{prop:L2} it would be enough to assume that 
 $\rho_0$ is smooth and compactly supported.  
 Then 
$\| D_t\|_{2,\psi}^2 $  and   $\| \partial_\eta D_t\|_{2,\psi}$   
are finite  for any $t >0$.  Hence 
\be{eq:d-Dt} 
\begin{split} 
\frac{1}{2}\frac{\dd}{\dd t}\| D_t\|_{2,\psi}^2 &= 
- \|\partial_\eta  D_t\|_{2,\psi}^2 + g(t) \langle \rho_t^g , \partial_\eta D_t\rangle_\psi - h(t) 
\langle   \rho_t^h , \partial_\eta D_t\rangle_\psi \\ 
&= 
- \|\partial_\eta  D_t\|_{2,\psi}^2 + g(t) \langle  D_t  , \partial_\eta D_t\rangle_\psi  + (g(t)-h(t)) 
\langle   \rho_t^h , \partial_\eta D_t\rangle_\psi \\
&\leq \frac{g(t)^2}{2} \| D_t\|_{2,\psi}^2 + \frac{(g(t) - h(t) )^2}{2} \|\rho_t^h\|_{2,\psi}^2 .
\end{split}
\ee
Since $D_0= 0$, the bound \eqref{eq:Dt-bound}  follows from 
\eqref{eq:L2-rhoht}. 
\end{proof}

 \subsection{Strong form of the  local   McKean-Vlasov equation.} 
 We can now formulate the McKean-Vlasov problem.
 To do so, we define the set $\AA$ of \emph{admissible drift fields}
 $h:\Tod\times \R_+\rightarrow \R $: 
 \be{drifty.2}
 \AA\equiv 
  \sfC^{0,\infty}\lb \Tod\times \R_+,\R \rb  .
 \ee
 Next, define the set $\BB$ of {\em admissible density fields}.
  \begin{definition}\label{mkprob.2}
  A density $\rho (x ,\cdot)$ is a \emph{  
 nice 
 profile }
  if it is smooth in $\eta$, and continuous in $x$  
  as a map from $\bbT^d$ to $L^2 \lb \bbR , \eee^{-{2}\psi} \rb$. 
  In particular, nice  profiles satisfy  
 $\max_x \| \rho  (x ,\cdot)\|_{2,\psi} <\infty$.

 We denote by $\BB$ the set  of density fields 
 $\rho: \Tod\times \R_+\times\R\mapsto \R$ which satisfy:
 \begin{itemize}
 \item[(i)]  $\rho_t (x , \cdot )$ is nice for any $t\in\R_+$. 
 \item[(ii)] For any $x\in \Td$, $\rho_t (x , \eta )$ are $\sfC^{\infty , \infty}\lb \R_+ 
 \times\bbR \rb$ in 
 $(t, \eta )$.  
 \end{itemize}
  \end{definition}
 
 \begin{remark} 
 Continuity in $x$ in the above definition of nice profiles is redundant, and we assume it for convenience 
 and for clarity of exposition.  Theorem~\ref{thm:L-MV} and Theorem~\ref{thm:HDL} hold if, instead
 of continuity,  one assumes measurability and boundness - $\sup_{x\in \bbT^d}  \|  \rho (x ,\cdot )\|_{{2} ,\psi} <\infty$.  
 \end{remark}
 
 \begin {definition}\label{mkprob.1}
 Given initial density $\rho_0 (x , \theta )$ a strong 
 solution of the McKean-Vlasov equation {\eqref{empdef.3}} is a pair 
 $(\rho,h)$, with $\rho \in \BB$ and $h\in \AA$, 
such that 
 \be{eq:L-MF-MV}
 \forall\ x\in\Tod\quad \text{and}\ \forall\ t\in
 { [0,\infty)}
 \begin{cases} 
 &\rho_t (x , \cdot ) = \rho_t^{h^x}\ \text{with initial condition $\rho_0 (x , \cdot )$} \\
 &\quad {\rm and} \\
 &h^x(t)  = \int\int J(y- x) \rho_t (y , \eta )\eta \eee^{-{2}\psi (\eta )}
 \dd\eta\dd y 
 \end{cases}
 \tag{L-MV} 
 \ee
 \end{definition}
 Above we continue to use  $\rho_t^h$ for the density at time $t$ of the 
 time inhomogeneous 
 diffusion with generator $L_{h(t)}$.

 The following theorem asserts the existence and uniqueness of the McKean-Vlasov problem.
 \begin{theorem} 
 \label{thm:L-MV}
 Assume that 
 $\rho_0 (x ,\cdot)$ is a nice initial profile. 
 Then
 there exists a
  unique strong 
 solution $(\rho,h)$
 of the 
 system \eqref{eq:L-MF-MV}. 
 \end{theorem}

An equivalent reformulation of the Theorem is to say  that  for any $T$ fixed a unique  strong solution 
exists on any time interval $[0,T]$. 
The proof of the latter  is based on contraction properties of the map 
$\Phi$ which we construct below. 

\subsection{The map $\Phi$.} 
\label{sub:Phi} 
Fix an initial density $\rho_0 (x , \theta )$ which satisfies the assumptions of 
Theorem~\ref{thm:L-MV}. 
We define a map $\Phi:\AA\rightarrow \AA$ by
\be{eq:Phi-map} 
\Phi [h]^x(t) \equiv \int_{\bbT^d} \int_\bbR J (y-x ) \eta \rho_t^{h^y} 
(\eta ) \eee^{-{2}\psi (\eta )}
\dd\eta\dd y .
\ee
It is useful to view this map as the composition of two maps, 
\be{m.1}
\Phi_1:\AA\rightarrow\BB,
\ee
where, for $h\in \AA$,
\be{m.2}
(\Phi_1(h))^x(t)\equiv \rho_t^{h^x},
\ee
and
\be{m.3}
\Phi_2:\BB\rightarrow\AA,
\ee
where, for $\rho\in \BB$,
\be{m.4}
(\Phi_2(\rho))^x(t)\equiv  \int_{\bbT^d} \int_\bbR J (y-x ) \eta \rho^y_t 
(\eta )
\eee^{-{2}\psi(\eta)}\dd\eta\dd y .
\ee
Clearly, $\Phi(h)=\Phi_2\circ\Phi_1(h)$.
The fact that $\Phi_1$ maps $\AA$ into $\BB$ follows from 
the proof of Lemma \ref{lem:prop-dens},
specifically from \eqref{eq:rhoth} and \eqref{eq:qth}, 
and from Lemmas
\ref{prop:L2}--\ref{prop:stability}. The fact that 
$\Phi_2$ maps $\BB$ into $\AA$ follows readily from its definition and the 
smoothness of $J$. 
Therefore the composite map $\Phi$ maps $\AA$ into $\AA$, i.e., if $h$ is an admissible drift field then 
 $\Phi [h]$ is also an admissible drift field.

As a first step we prove the following a priori bounds.
\begin{lemma} 
\label{lem:Sup-Phin}
Assume that  $h\in \AA$   satisfies 
\be{eq:asn-htx} 
\|h\|_\infty\df \sup_t\max_x \abs{h^x(t)}  <\infty .
\ee
Then, 
\be{eq:sup-Phin}
|h|_\Phi \df \sup_n \left\|\Phi^n [h]\right\|_\infty  <\infty. 
\ee
\end{lemma}

\begin{proof}
It is enough to  prove the statement of Lemma~\ref{lem:Sup-Phin} for 
one-sided quantities 
\be{one.1}
 \abs{h}_+\df \sup_t\max_x h^x(t)\quad {\rm and}\quad 
 |h|_{\Phi ,+} \df \sup_n \abs{\Phi^n [h]}_{+} .
\ee
Recalling \eqref{eq:rhoth} and \eqref{eq:Heat-k}, we see that
 \be{one.2}
  \Phi [h]^x(t) = \int_{\bbT^d} \int_\bbR J (y-x )  \rho_0 (y , \theta ) \eee^{-{2}\psi (\theta )} 
  \bbE_\theta^{h^y}\lb \theta(t)\rb\,  
  \dd\theta \dd y .
 \ee
 The idea behind the proof is that if $|h|_+$ is very large, then the strong 
 inward drift due to the potential $\psi$ will ensure that 
  $  \bbE_\theta^{h^y}\lb \theta(t)\rb\ll |h|_+$, which in turn implies that 
 $  \left|\Phi [h]\right|_+$ will be smaller than $|h|_+$. This implies that 
  $\left|\Phi^n [h]\right|_+$ cannot grow indefinitely with $n$, which is the 
  assertion of the theorem. 
  
To prove this, define 
\be{eq:theta-star} 
\theta^* (h ) = \sup\lbr \theta~:~ \exists t, y\ 
\text{such that $\bbE_\theta^{h^y} ( \theta(t))  \geq \theta$}\rbr \vee 0.
\ee
Since by  assumption  
\eqref{eq:psi-con1} on $\psi$, the derivative $\psi^\prime (\theta )$ tends to $\infty$ as
$\theta$ grows,  
it is easy to deduce from comparison results for one dimensional diffusions  
(the proof of \eqref{eq:theta-star-lim} below  gives a quantitative bound along these lines)
that $\theta^* (h )$ is finite whenever $\abs{h}_+ < \infty$.
Then,  since for any $\theta$ and $t$,  $\bbE_\theta^{h^y} ( \theta(t)) 
\leq \abs{\theta}\vee \theta^* (h )$, 
\be{eq:abs-Phi-bound} 
\abs{\Phi [h]}_+ \leq \theta^* (h ) \hat J_0 + c(\rho_0 ) , 
\ee
where $\hat J_0 {:= \int_{\Tod} \abs{J (x)}\dd x}$, and 
\be{eq:c-rho} 
c (\rho_0 ) = \max_x \int_{\Tod}\int_{\bbR} J (y-x ) \rho_0 (y , \theta )\abs{\theta} \eee^{-{2}\psi (\theta )}\dd \theta \dd y 
\leq \hat J_0 \|\theta\|_{2,\psi} \max_x \|\rho_0 (x , \cdot ) \|_{2,\psi} .
\ee
The term $c(\rho_0)$ is just a finite constant and does not depend on 
$h$, and hence is irrelevant.
What we need to show is that $\theta^*(h)$ becomes much smaller than 
$|h|_+$, as $|h|_+$ grows, i.e. that 
\be{eq:theta-star-lim} 
\lim_{H\to\infty} \sup_{\abs{h }_+ = H} \frac{\theta^* (h )}{H} = 0 .
\ee
Clearly, from the discussion above, \eqref{eq:theta-star-lim} implies that 
$\lbr \abs{\Phi^n [h]}_+\rbr$ is a 
bounded sequence, and \eqref{eq:sup-Phin} follows.

It remains to prove \eqref{eq:theta-star-lim}.  The argument is based on 
the following comparison result for one-dimensional 
diffusions:  For $i=1,2$ consider 
\be{eq:dif-bi} 
\dd\theta^i_t = \dd B(t)  + b^i (t , \theta^i_t )\dd t .
\ee 
Assume that the fields $b^i$ are continuous in $t$ and smooth in $\theta$ and that 
\be{eq:bi-cond} 
\forall ~t, \theta \quad  b^1 (t, \theta ) 
\leq b^2 (t , \theta )\quad 
\text{and $\limsup_{\abs{\theta}\to\infty}\sup_t { {\rm sign}(\theta )  
b^2 (t, \theta )} < 0 $} .
\ee
Then,  for any $t$ and $\theta$, 
\be{eq:theta-i-conc} 
\bbE_\theta (\theta^1_t ) \leq \bbE_\theta (\theta^2_t ). 
\ee
Indeed, let $\tau_n$ be the first exit time from $[-n, n]$. The second condition in  \eqref{eq:bi-cond} implies that 
$\bbE_\theta \lb \theta^2_t\rb = \lim_{n\to\infty} \bbE_\theta \lb \theta^2_{t\wedge \tau_n}\rb$. On the other hand,  
by the usual Yamada comparison result; {see for instance Porposition~5.2.18 in \cite{KS91},} 
\be{tttt.5}
\bbE_\theta (\theta^1_{t\wedge \tau_n} ) \leq \bbE_\theta (\theta^2_{t\wedge\tau_n} ), 
\ee
for any $n, t$ and $\theta$.  
Consequently, in order to prove \eqref{eq:theta-star-lim}  we may substitute 
$-\psi^\prime (\theta ) + h(t)$ by a larger drift 
$b(t, \theta )$ which satisfies the second condition in 
\eqref{eq:bi-cond}, 
and, furthermore, 
we may choose some $\bar\theta $ and consider reflected diffusions on  
$[\bar\theta , \infty )$. 
The assumption  \eqref{eq:psi-con1} on $\psi$ implies that there exists 
a $\theta_0\in \R_+$ such that $\psi$ is striclty convex on
 $[\theta_0,\infty)$. Consequently, if $|h|_+=H$, there exists 
 $\eta\in \R_+$, such that for all
 $\theta \in [\eta,\infty)$, 
 \be{convex.1}
 -\psi^\prime (\theta ) + h(t)\leq -\psi^\prime (\theta ) + H\leq 
 -1.
 \ee
For instance \eqref{convex.1} holds if we choose 
$\eta \sim \sqrt[2k-1]{H}$.
If we denote by $\bar\theta^{{\eta}}_t$ the reflection at $\eta$ of the 
Brownian motion with unit negative drift, then 
\be{eq:refl-dif} 
\bbE_\theta \lb \bar\theta^\eta_t \rb  \geq \bbE_\theta^h (\theta(t) ) .
\ee
for any $\abs{h}_+\leq H$ and for any  $t\geq 0, \theta \geq \eta$. 

The reflected diffusion $\bar \theta^\eta$ is positively recurrent and its 
invariant distribution 
has density $f_\eta (\theta )= 2\eee^{-2 (\theta -\eta )}
\1_{\lbr \theta >\eta\rbr}$. 
Hence, by monotone coupling, 
\be{gauss.1}
\sup_t \bbE_\eta \lb \bar\theta^\eta_t \rb  
\leq
\int_{\eta}^\infty \theta f_\eta (\theta )\dd\theta = 
\eta + \frac{1}{2}. 
\ee
Consequently, for any  $t$ and for any $\theta >\eta$, 
\be{gauss.2}
\begin{split}
\bbE_\theta\lb \bar\theta^\eta_{t}\rb & = 
\bbE_\theta\lb \bar\theta^\eta_{t\wedge \tau_\eta} \rb +  
\bbE_\theta\lb \bar\theta^\eta_{t} - \bar\theta^\eta_{t\wedge \tau_\eta}\rb = 
\bbE_\theta\lb \bar\theta^\eta_{t\wedge \tau_\eta} \rb +  
\bbE_\theta\lb \bar\theta^\eta_{t} - \bar\theta^\eta_{t\wedge \tau_\eta}\rb 
\1_{\lbr \tau_\eta \leq t\rbr} 
\\ 
& 
= 
\lb\theta - \bbE_\theta\lb t \wedge \tau_\eta\rb\rb
+ \bbE_\theta\lb \bar\theta^\eta_{t} - \eta \rb 
\1_{\lbr \tau_\eta \leq t\rbr}  \stackrel{\eqref{gauss.1}}{\leq}  
\lb\theta - \bbE_\theta\lb t \wedge \tau_\eta\rb\rb  + \frac{1}{2}  
\bbP_\theta\lb \tau_\eta \leq t\rb. 
\end{split}
\ee
However,  $- \bbE_\theta\lb t \wedge \tau_\eta\rb + \frac{1}{2}
\bbP_\theta\lb \tau_\eta \leq t\rb $ is non-positive for all $t\geq 0$ as 
soon as $\theta- \eta \geq 
K_1
$
some sufficiently large constant $K_1$.  
In particular, if we chose
 $\eta \sim \sqrt[2k-1]{H}$,  
we conclude that there exists a finite constant $K_2$,  such that 
 $\bbE_\theta \lb \bar\theta^\eta_t \rb <  \theta$, 
for all $t>0$ and all $\theta>K_2\sqrt[2k-1]{H}$.
The target \eqref{eq:theta-star-lim} follows.  
\end{proof}

Differentiating  both sides of  
\eqref{eq:Phi-map}  with respect to $t$ and relying on  \eqref{eq:FP-Lh}, we observe:  
\be{eq:nabla-t1} 
\begin{split}
\partial_t \Phi^{n+1} [h]^x(t)  &= \int_{\bbT^d} \int_\bbR J (y-x ) \lb L_{\Phi^n [h]^y(t)} \eta\rb   \rho_t^{\Phi^n [h]^y} (\eta ) \eee^{-{2}\psi (\eta )}\dd\eta\dd y \\ 
&= 
\int_{\bbT^d} \int_\bbR J (y-x )  \lb \Phi^n [h]^y(t) - \psi^\prime (\eta )\rb 
\rho_t^{\Phi^n [h]^y} (\eta ) \eee^{-{2}\psi (\eta )}\dd\eta\dd y . 
\end{split}
\ee
We can continue differentiating with respect to $t$ in \eqref{eq:nabla-t1} .   For instance, 
\be{two.1}
\partial_t ^2 \Phi^{n+1} [h]^x(t) = 
 \int_{\bbT^d} \int_\bbR J (y-x )  \lb \partial_t \Phi^n [h]^y(t) -  L_{\Phi^n [h]^y(t)}  \psi^\prime (\eta )\rb 
\rho_t^{\Phi^n [h]^y} (\eta ) \eee^{-{2}\psi (\eta )}\dd\eta\dd y . 
\ee
By our assumption \eqref{eq:psi-con1} all integrals of the form $\int_\bbR \lb \psi^{(k)} (\eta )\rb^\ell \eee^{-{2}\psi (\eta )} \dd\eta$ 
are finite. 
By iteration, and in view of \eqref{eq:L2-rhoht}, we arrive to the following 
conclusion: 
\begin{lemma} 
\label{lem-deriv-t} 
For every  $k\in \mathbb{N}$
there exists a monotone function  $c_k$ on $[0,\infty)^2$, such 
that the following happens:  If 
 an admissible drift field $ h\in \AA$ satisfies 
\eqref{eq:asn-htx}  
and hence, by Lemma~\ref{lem:Sup-Phin} also \eqref{eq:sup-Phin}, 
 then, for every $T\geq 0$ and each $n\geq k$, 
\be{eq:der-bounds} 
\max_{s\leq T} \max_{y\in \bbT^d} \left|\partial^k_s \Phi^n [h]^y(s)\right| 
\leq c_k (T, |h|_\Phi ) .
\ee
\end{lemma}


The next lemma shows that $\Phi$ is a contraction for short times.

\begin{lemma}\label{banach1}
Let $h,g\in \AA$. Then, for any $t\in \R_+$, 
\be{eq:Phi-gh-step} 
\| \Phi [g](t)  - \Phi [h](t)  \|^2_{L^2 (\bbT^d )}  
\leq C_J^\prime  \eee^{ (|g |_\Phi^2 + |h|_\Phi^2 )t}  
\int_0^t  \| g(s) - h(s)\|_{L^2 (\bbT^d )}^2 \dd s , 
\ee
with 
\be{gauss.4}
C_J^\prime =  \sup_{x\in\bbT^d} \|\rho_0^x\|_{2,\psi}^2  
\int_{\bbT^d} \int_\bbR  J^2 (y ) \eta^2 \eee^{-{2}\psi (\eta )}\dd\eta \dd y .
\ee
\end{lemma}
\begin{proof}
First, for any $x$ and for any $t>0$, using the Cauchy-Schwartz inequality, we obtain that 
\be{eq:Phi-gh}
\left| \Phi [h ]^x(t) -  \Phi [g ]^x(t)\right|^2  \leq  C_J \int_{\bbT^d} \left\| \rho_t^{g^y } - \rho_t^{h^y}\right\|^2_{2,\psi} \dd y  , 
\ee
where 
\be{gauss.3}
C_J = \int_{\bbT^d} \int_\bbR  J^2 (y ) \eta^2 \eee^{-{2}\psi (\eta )}\dd\eta \dd y .
\ee
By Lemma \ref{prop:stability},
\be{banach.2}
\left\| \rho_t^{g^y } - \rho_t^{h^y}\right\|^2_{2,\psi}
\leq \eee^{\int_0^t \left(g^y({s})^2+h^y(s)^2\right)\dd s}\left\|\rho^y_0\right\|_{2,\psi}^2\int_0^t\left(g^y(s)-h^y(s)\right)^2\dd s.
\ee
We now use Lemma \ref{lem:Sup-Phin} to bound the right-hand side by
\be{babach.3}
 \eee^{t \left( |g|_{\Phi}^2+ |h|_{\Phi}^2\right)}
 \left\|\rho^y_0\right\|_{2,\psi}^2\int_0^t\left(g^y(s)-h^y(s)\right)^2\dd s.
 \ee
Integrating the resulting bound over the torus yields the assertion of the lemma.
\end{proof}

We are now ready to proof the theorem. 

\begin{proof} (of Theorem \ref{thm:L-MV})
Let $h , g\in \AA$ be any two initial 
drift fields 
satisfying  $\|h\|_\infty , \|g\|_\infty <\infty$. 
We want to show that the sequences $\Phi^n[h]$  and  
 $\Phi^n[g]$
converge  to 
the same fix-point of $\Phi$.  
Iterating the bound in \eqref{eq:Phi-gh-step} gives 
\be{eq:Phi-gh-nstep} 
\sup_{s\leq t}\|\Phi^{n}  [h](s)  - \Phi^{n}  [g](s)  \|^2_{L^2 (\bbT^d )}  
\leq  \frac{\lb t C_J^\prime  \eee^{ (|g |_\Phi^2 + |h|_\Phi^2 )t}  \rb^n }{n !} 
\sup_{s\leq t } \| h (s )  - g(s)\|_{L^2 (\bbT^d )}^2  .
\ee
Therefore, if $\hat h = \lim_{n\to\infty} \Phi^{n}  [h]$ exists for {\em some} 
$\|\cdot\|_\infty$-bounded $h\in \AA$, then  
$\hat h = \lim_{n\to\infty} \Phi^{n}  [g]$ for {\em any} 
$\|\cdot\|_\infty$-bounded bounded $g\in \AA$. That is existence would imply 
uniqueness.

However, plugging  a $\|\cdot\|_\infty$-bounded $h\in \AA$ and  
$g = \Phi[h]$  into \eqref{eq:Phi-gh-nstep} implies 
that $\Phi^n[h]$ is a Cauchy sequence,  
and hence 
indeed converges to some $\|\cdot\|_\infty$-bounded element 
$\hat h\in\AA$, such that for any $t <\infty$, 
\be{eq:Phi-lim} 
\lim_{n\to\infty} \sup_{s\leq t}\| \Phi^n [h](s) -\hat h(s)\|_{L^2 (\bbT^d )}^2 \dd s = 0 .
\ee
Moreover, $\hat h$ is a fixpoint of $\Phi$. 
Therefore $\hat\rho = \Phi_1 (\hat h)$ satisfies $\hat h = \Phi_2 (\hat\rho )$. This is precisely \eqref{eq:L-MF-MV}.
By \eqref{eq:rhoth} and Lemma~\ref{lem-deriv-t}, $\hat\rho_t (x , \eta )$ is 
continuous in $x$ and 
 smooth in $\eta$ and $t$.
\end{proof}

\section{Hydrodynamic Limit (HDL)}
\label{sec:HDL}
Having established the existence of a unique smooth solution to the McKean-Vlasov equations, we 
now prove the convergence of the empirical process of our particle system to this solution. The
proof is based on entropy estimates.

Let $\rho_0 (x , \theta )$ be a nice initial profile, and let 
$\lb \rho , h \rb$ 
be  
the unique  strong 
solution to \eqref{eq:L-MF-MV} with initial condition $\rho_0$.
{In the sequel, elements of $\bbR^{\N^d}$ will be denoted as $\utheta$. } 
For each $N$ and $t\geq 0$  consider the following product density on 
$\bbR^{N^d}$ with respect to 
${\rm e}^{-\sum{2}\psi (\theta_i )} \df 
{\rm e}^{-{2}\psi(\utheta )}$: 
\be{eq:rho-N}
\rho_t^N (\utheta ) = \prod_1^{N^d} \rho_t \lb\frac{i}{N} , \theta_i \rb.
\ee
Another way to think about $\rho_t^N $ is as follows: For each $i=1, \dots, N$ 
define $h^i_t\df h^{i/N}_t$, where  $h_t^x$ satisfies the second of \eqref{eq:L-MF-MV}.

Consider 
\be{eq:SDE-N-hat}
\dd\hat \theta_i (t ) = \dd B_i (t) -\psi^\prime \lb \hat\theta_i (t )\rb \dd t+
h^i_t\dd t .
\ee
Then $\rho_t \lb\frac{i}{N} , \theta \rb $ is the density of $\hat\theta_i (t)$.

Let us turn to the microscopic dynamics \eqref{eq:SDE-N}. 
Let $f_0^N (\utheta )$ be the initial density of $\utheta (0) $ on
 $\bbR^{\TodN}$ 
with respect to 
$\eee^{-2\psi(\utheta )}$. 
We assume that the relative entropy  
\be{funny.1}
\calH \lb f_0^N\, \big |\,  \rho_0^N\rb := \int f_0^N (\utheta )\ln\lb\frac{ f_0^N (\utheta )}{\rho_0^N (\utheta }\rb  
\eee^{-2\psi(\utheta )}\dd \utheta , 
\ee
satisfies
\be{eq:C}
\lim_{N\to\infty}\frac{1}{N^d}\calH \lb f_0^N\, \big |\,  \rho_0^N\rb = 0 .
\tag{C}
\ee

Our first proposition states that this property is conserved in time.
\begin{proposition}
\label{thm:A} 
 Let $\bbP^N_T$ be the distribution on $\sfC\lb [0,T], \bbR^{N^d}\rb$ of the 
 diffusion 
 process \eqref{eq:SDE-N} with initial density $f^N_0$, and let $\hat\bbP^N_T$ be 
 the distribution
 of the decoupled process \eqref{eq:SDE-N-hat} with the 
 initial density  $\rho^N_0$. Then, 
 assuming \eqref{eq:C}
 \be{eq:ClaimA}
 \lim_{N\to\infty}\frac{1}{N^d}\calH \lb \bbP^N_T \, \big| \, \hat\bbP^N_T\rb = 0 , 
 \ee
 for any $T \geq 0$. 
\end{proposition}
The proof of this proposition is given in the next subsection.

Consider now the empirical profiles $\mu^N_t$ defined in
 \eqref{empdef.1}.
For each $t$ the random measure $\mu^N_t$ is a probability 
measure on $\bbR\times \Tod$. We 
denote the latter space as $\bbM_1 (\Tod \times\bbR  )$. It is a Polish space, 
and in the sequel we shall use $\dd_{\mathsf{LP}}$ for its Lévy-Prohorov type 
metric, 
\be{eq:LP-met} 
\dd_{\mathsf{LP}} (\mu , \nu ) := \sup_{f\in \mathsf{Lip}\lb \Tod\times \bbR \rb} 
\left| \int f (x , \theta )\mu \lb \dd x,\dd\theta \rb -\int f (x , \theta )\nu  \lb \dd x,\dd\theta \rb 
\right| , 
\ee
where $\mathsf{Lip}\lb \Tod\times \bbR \rb$ is the set of all globally $1$-Lipschitz functions $f$ 
with $\|f \|_{\infty} \leq 1$. 

For $t\leq T$ we can think of $\mu^N_t$ in terms of 
marginals 
of the random measures on path space: 
For each $N$ and every  $i\in\TodN$, the trajectory  $\theta_i^N := \theta_i^N [0,T]$ in 
\eqref{eq:SDE-N} is a 
random element of $\sfC \lb [0,T]\rb$. Define
\be{eq:L-N} 
L^N_T (\dd x , \dd\theta ) = 
\frac{1}{N^d} \sum_i\delta_{\lb i/N ,  \theta_i^N [0,T]\rb }.
\ee
Then, 
$L^N_T$ is a random element of $\bbM_1 \lb \Tod\times \sfC \lb [0,T] ,\bbR\rb \rb$. The map 
\be{plop.7}
 L^N_T\ \mapsto\ \mu^N [0,T] := \lb \mu^N_t\, ;\, t\in [0,T]\rb 
\ee
is continuous from $\bbM_1 \lb \Tod\times \sfC \lb [0,T] ,\bbR\rb \rb$ to 
$\sfC \lb [0,T] , \bbM_1 (\bbR\times\Tod )\rb$. 

With a slight abuse of notation we continue to use $\hat\bbP^N_T$
and, respectively,  $\bbP^N_T$ for the distributions of   $L^N_T$ 
on $\bbM_1 \lb \Tod\times \sfC \lb [0,T] ,\bbR\rb \rb$
and $\mu^N [0,T]$ on $\sfC \lb [0,T] , \bbM_1 (\bbR\times\Tod )\rb$, whenever 
$\utheta$ is the decoupled diffusion \eqref{eq:SDE-N-hat} or,
 respectively, if 
$\utheta$ satisfies the local mean-field sde 
\eqref{eq:SDE-N}.

Our next proposition states 	an exponential concentration bound 
for the {\em decoupled} measure $\hat\bbP_T^N$.

\begin{proposition} 
\label{thm:conc-PN-hat} 
Let $\rho_0$ be a nice initial profile. Then for any 
$T<\infty$  
 and $\epsilon >0$ 
there exists a  positive 
constant 
$C_T (\epsilon ) > 0$, such that 
\be{eq:conc-PN-hat}
\hat\bbP_T^N\lb \max_{t\in [0, T]}\dd_{\mathsf{LP}}
\lb \mu^N_t , \rho_t (x, \theta ) {\rm e}^{-{2}\psi (\theta )}\dd x\dd\theta \rb \geq \epsilon \rb 
\leq {\rm e}^{-N^d C_T (\epsilon )} ,
\ee
for all  $N$  large enough. 
\end{proposition}
This proposition will be proven in Subsection 3.2 below.

We have now the tools to prove convergence to the hydrodynamic limit.
\begin{theorem} 
 \label{thm:HDL}
Let $\rho_0$ be a nice initial profile. 
 Under Assumption~\eqref{eq:C} the distribution $\bbP_T^N$ of $\mu^N [0,T]$, on 
 $\sfC \lb [0,T] , \bbM_1 (\bbR\times\Tod )\rb$ 
 converges 
 to $\delta_{\rho_\cdot (x , \theta ){\rm e}^{-{2}\psi (\theta )}\dd x\dd\theta}$ in the 
 following sense: For any $\epsilon >0$ and $T<\infty$, 
 \be{eq:HDL} 
 \lim_{N\to\infty} 
 \bbP_T^N\lb \max_{t\in [0, T]}\dd_{\mathsf{LP}}
\lb \mu^N_t , \rho_t (x, \theta ) {\rm e}^{-{2}\psi (\theta )}\dd x\dd\theta \rb \geq \epsilon \rb = 0.
 \ee
\end{theorem}

\begin{proof} 
The entropy inequality (c.f. \cite{Yau91})  states that for any event $A$ 
or, more generally, for any random  variable $X$, 
\be{entropy.1}
 \bbP_T^N \lb A\rb  \leq 
 \frac{\log 2 + \calH \lb \bbP^N_T \, \big|\,  \hat\bbP^N_T\rb}
 {\log\lb 1+ 1/ \hat \bbP_T^N (A)\rb},
 \ee
 respectively
 \be{entropy.1.1}
 \bbE_T^N \lb X \rb \leq \calH \lb \bbP^N_T \, \big|\,  \hat\bbP^N_T\rb +
 \log \hat\bbE^N_T \lb {\rm e}^{X}\rb .
\ee
Using this with $A = A_N$ the event considered in 
\eqref{eq:conc-PN-hat}, and inserting the 
assertions of Proposition~\ref{thm:A} and 
Proposition~\ref{thm:conc-PN-hat} into the right-hand side of
this inequality immediately yields \eqref{eq:HDL}. Namely, by Proposition~\ref{thm:conc-PN-hat}, 
 $\log\lb 1+ 1/ \hat \bbP_T^N (A_N)\rb \geq N^d C_T (\epsilon )$, and 
 \eqref{eq:ClaimA} applies. 
\end{proof}

It remains to prove Proposition~\ref{thm:A} and Proposition~\ref{thm:conc-PN-hat}. This is the content of
 Subsection~\ref{sub:thmA} and Subsection~\ref{sub:thmconc-PN-hat}, respectively. 
 
\subsection{Proof of Proposition~\ref{thm:A}} 
\label{sub:thmA}
Set 
\be{eq:ent-H}
 h^i (\utheta ) = \frac{1}{N^d} \sum_j J\lb \frac{j-i }{N}\rb  \theta_j .
\ee
By Girsanov's formula, 
 \be{plop.8}
  \calH \lb \bbP^N_t \, \big| \, \hat\bbP^N_t\rb = \calH  \lb f_0^N\big | \rho_0^N\rb 
  + \frac{1}{2}\bbE_t^N\sum_i \int_0^t \lb h^i (\utheta_s ) - h^i_s\rb^2 \dd s .
 \ee
 Hence, 
 \be{eq:der-H}
  \frac{\dd }{\dd t} \calH \lb \bbP^N_t \, \big|\,  \hat\bbP^N_t\rb = 
  \frac{1}{2}\bbE_t^N\sum_i  \lb h^i (\utheta_t ) - h^i_t\rb^2 . 
 \ee
Since  $(\rho , h)$ is a strong solution to \eqref{eq:L-MF-MV}, both $\rho$ and $h$ are
continuous in $x$. Recall that $J$ is also assumed to be continuous. It follows that
\be{plop.8-1}
  h^i_t =  \lbr \frac{1}{N^d} \sum_j 
  J\lb \frac{j-i }{N}\rb  \hat\bbE_T^N \lb \theta_j (t )\rb \rbr  + \smo{1} 
  = \hat\bbE_T^N \lb h^i (\utheta_t )\rb  +\smo{1}, 
\ee
uniformly in $i\in\TodN$. 
Define 
\be{eq:X-var}
\eta_t^i = \theta_i (t ) - \hat\bbE_T^N  (\theta_i (t ) ) 
\quad {\rm and}\quad X_t = \frac{1}{N^{2d} }\sum_{i,j} K\lb\frac{i-j }{N}\rb
\eta_t^i\eta_t^j, 
\ee
where 
\be{plop.9}
 K\lb\frac{i-j }{N}\rb = \frac{1}{N^d}\sum_\ell J\lb \frac{i-\ell}{N}\rb J\lb \frac{j-\ell}{N}\rb .
\ee
Using  \eqref{plop.8-1} to approximate  $h_t^i$, we infer that with the 
notation above, 
\be{plop10-1}
 \bbE_t^N\sum_i  \lb h^i (\utheta_t ) - h^i_t\rb^2 = \bbE_t^N \lb N^d X_t \rb + N^d\smo{1} .
\ee
Consequently, 
\eqref{eq:der-H} reads as
\be{plop.10}
 \frac{\dd }{\dd t} \calH \lb \bbP^N_t \, \big|\,  \hat\bbP^N_t\rb = 
 \frac{1}{2}\bbE_t^N \lb N^d X_t \rb + N^d\smo{1} .
\ee
By the entropy inequality {\eqref{entropy.1}},  for any $\delta >0$, 
\be{eq:ent-bound} 
\frac{\dd }{\dd t} \calH \lb \bbP^N_t \, \big|\,  \hat\bbP^N_t\rb \leq 
\frac{1}{2\delta} \calH \lb \bbP^N_t\,  \big| \, \hat\bbP^N_t\rb + 
{\frac{1}{2\delta}}\log\hat\bbE_t^N\lb {\rm e}^{\delta N^d X_t}\rb + N^d\smo{1} .
\ee
Under $\hat \bbP_t^N$ variables $\eta_t^i$ in \eqref{eq:X-var} are 
independent and centred. Furthermore, by \eqref{eq:L2-rhoht} and Theorem~\ref{thm:L-MV}
the densities (with respect to ${\rm e}^{-{2}\psi (\theta )}$) $q_t^i (\theta )$ of $\eta_t^i$ 
satisfy the following property: There exists a finite constant 
$C = C (t, \rho_0 )<\infty $ such that
\be{eq:max-L2} 
\max_{i}\| q_t^i\|_{2, \psi} \leq C .
\ee
In such circumstances the following holds.
\begin{lemma}
 \label{lem:small-delta} 
 For each $t>0$ there exists $\delta >0$ such that
 \be{eq:small-delta} 
 \lim_{N\to\infty}\frac{1}{N^{d}} \log\hat\bbE_t^N\lb {\rm e}^{\delta N^d  X_t}\rb = 0,
 \ee
 uniformly in $t\leq T$. 
\end{lemma}
The claim of Theorem~\ref{thm:A} is now  straightforward. 
It remains to prove Lemma~\ref{lem:small-delta}. 

\begin{proof}[Proof of Lemma~\ref{lem:small-delta}.] 
In view of our basic assumption \eqref{eq:psi-con1}, the uniform 
bound \eqref{eq:max-L2} implies that random 
variables $\eta_t^i$ are uniformly sub-Gaussian: there 
exists $\sigma = \sigma_t <\infty$, such that,
for any $\alpha\in \bbR$, 
\be{eq:uni-SG} 
\max_{i\in{\TodN}} {\hat\bbE_t^N}\lb {\rm e}^{\alpha \eta_t^i} \rb \leq {\rm e}^{\sigma_t \alpha^2} .
\ee
Furthermore, there exists $\kappa_0 >0$ and a finite convex function $g_t $ on $(-\kappa_0 , \kappa_0 )$ 
with $g_t (0) = 0$ 
such that,  for any $\abs{\kappa} \leq \kappa_0$, 
\be{eq:uni-kappa-square} 
\max_{i\in {\TodN}} {\hat\bbE_t^N} \lb {\rm e}^{\kappa ( \eta_t^i)^2 } \rb \leq {\rm e}^{g_t (\kappa ) } .
\ee 
Since we care only about small $\delta$ in \eqref{eq:small-delta} we can rescale 
both the variables $\eta_t^i\mapsto \epsilon \eta_t^i$
and the kernel $K \mapsto \epsilon K$ and assume that \eqref{eq:uni-SG} holds with $\sigma_t =1$, 
that \eqref{eq:uni-kappa-square} holds with $\kappa_0 =1$ and also assume that $\max_x \abs{K (x )} \leq 1$. 
Then, \eqref{eq:small-delta} is a consequence of the following, ostensibly more general, statement: 
Let $g$ be a finite convex function on $[-1, 1]$ with $g (0) = 0$. 
Let $\eta_1, \eta_2, \dots , $ be independent centred random variables such that, for any $\alpha\in \bbR$
and for any $\abs{\kappa} \leq 1$, 
\be{eq:eta-cond} 
\sup_{{i\in \bbN}} \bbE {\rm e}^{\alpha\eta_i } \leq {\rm e}^{\alpha^2}\quad {\rm and}\quad 
\sup_{{i\in \bbN}} \bbE {\rm e}^{\kappa\eta_i^2  } \leq {\rm e}^{ g (\kappa )}
\ee
Finally let $K(i, j)$ be a matrix satisfying 
\be{eq:K-cond} 
\sup_{i, j\in\bbN}\abs{K (i, j )} \leq 1 .
\ee
Then 
\be{eq:small-delta-gen}
\limsup_{n\to\infty}\frac{1}{n} \log\lb \bbE {\rm e}^{\frac{\delta}{n}
\sum_{i, j=1}^n K(i, j )\eta_i\eta_j}\rb 
\leq 0.
\ee
for all $\delta$ sufficiently small. 

Indeed, define 
\be{eq:an-delta} 
\fra_n (\delta ) = \sup_{0\leq \nu \leq \delta}\sup_{\max\abs{K (i, j )}\leq 1} 
\bbE {\rm e}^{\frac{\nu}{n}\sum_{i, j=1}^n K(i, j )\eta_i\eta_j} .
\ee
Since for any kernel $K$, 
\be{plop.11}
 \bbE {\rm e}^{\frac{\nu}{n}\sum_{i, j=1}^n K(i, j )\eta_i\eta_j} = 
 \bbE_{\eta_1, \dots \eta_{n-1}} 
 \lb 
 {\rm e}^{\frac{\nu (n-1)}{n}\frac{1}{n-1}\sum_{i, j=1}^{n-1}
 K(i, j )\eta_i\eta_j} 
 \bbE_{\eta_n}
 \lb 
 {\rm e}^{\frac{ \nu K (n, n )}{n}\eta_n^2 + \lb \frac{2 \nu}{n}\sum_1^{n-1} K (i, n)\eta_i\rb \eta_n}
 \rb 
 \rb 
\ee
By Cauchy-Schwarz and  \eqref{eq:eta-cond}, 
\be{plop.12}
 \bbE_{\eta_n}
 \lb 
 {\rm e}^{\frac{ \nu K (n, n )}{n}\eta_n^2 + \lb \frac{2 \nu}{n}\sum_1^{n-1} K (i, n)\eta_i\rb \eta_n}
 \rb \leq 
 {\rm e}^{\frac{1}{2}\lb \frac{4 \nu}{n}\sum_1^{n-1} K (i, n)\eta_i\rb^2} {\rm e}^{\frac{1}{2}g \lb\frac{2\nu}{n}\rb} .
\ee
Consider the $(n-1)\times (n-1)$ kernel 
\be{eq:R-kernel} 
R(i, j ) = \frac{\nu (n-1)}{n} K (i, j ) + K (i, n )K (j, n) \frac{8\nu^2 (n-1)}{n^2} .
\ee
Since by assumption $\max_{i, j}\abs{K (i, j )} \leq 1$, clearly 
$\max_{i, j}\abs{R (i, j )} \leq 1$ as well, 
for all  $\nu$ small enough and uniformly in $n\in\bbN$. 
We therefore conclude that,  for all sufficiently small values of $\delta$, 
\be{eq:an-bound} 
\fra_n (\delta ) \leq \fra_{n-1} (\delta ) {\rm e}^{\max_{0\leq \nu\leq \delta} 
\frac{1}{2}g \lb\frac{2\nu}{n}\rb} 
{\leq\dots \leq a_1 (\delta ){\rm e}^{\frac{1}{2}\sum_{k=2}^n \max_{0\leq \nu\leq \delta} g \lb\frac{2\nu}{k}\rb}}
\ee
Since $g$ is continuous and zero at $0$,  
\be{eq:frac1-n-sum} 
\lim_{n\to\infty} \frac{1}{n}\sum_{{k}=2}^n \max_{0\leq \nu\leq \delta} g \lb\frac{2\nu}{{k}}\rb = 0. 
\ee
Hence \eqref{eq:small-delta-gen} holds. 
\end{proof}
\subsection{Proof of Proposition~\ref{thm:conc-PN-hat}} 
\label{sub:thmconc-PN-hat}
\eqref{eq:conc-PN-hat} is a rough bound. 
In Section \ref{sec:LD} below we shall discuss sharp large deviation estimates 
based on martingale techniques \cite{KO90}, 
 see also Section~4.2.1 in \cite{GA04}.

If $\rho_0$ is a nice initial 
profile
then, for any $T<\infty$,  the sequence of distributions 
$\hat\bbP_T^N$ of $L_T^N$ on 
 $\bbM_1 \lb \Tod\times \sfC \lb [0,T] ,\bbR\rb \rb$ and, consequently of 
 $\mu^N [0,T]$ on $\sfC \lb [0,T] , \bbM_1 (\bbR\times\Tod )\rb$ is exponentially 
 tight. 
 
 Indeed fix $C<\infty$ and consider the family $\calF_{C, T}$ of 
 one-dimensional diffusions 
 \be{eq:FC-fam} 
 \dd \theta (t ) = \dd B (t ) - \psi^\prime\lb \theta (t )\rb \dd t + h_t\dd t , 
 \ee
 with initial condition $\bbP\lb \theta (0) \in\dd \theta\rb = \rho (\theta ){\rm e}^{-{2}\psi (\theta )}$, 
 such that $h_t$ is smooth and 
 \be{eq:FC-cond} 
 \max_{0\leq t\leq T}\abs{h_t} \leq C\quad {\rm and}\quad \|\rho\|_{2, \psi} \leq C .
 \ee
 We can parametrise elements of  $\calF_{C, T}$ in terms of distributions $\bbP^{ h, \rho}_T$ on 
 $\bbM_1 \lb \sfC \lb [0,T] ,\bbR\rb \rb$, where $(h , \rho )$ satisfies \eqref{eq:FC-cond}. 
 We shall record
 this as $(h , \rho )\in \calF_{C, T}$. Then the family $\lbr \bbP^{ h, \rho}_T\rbr_{(h , \rho )\in \calF_{C, T}}$
 is uniformly tight, that is, for any $\epsilon > 0$, there exists a compact subset 
 $K_\epsilon\subset \sfC \lb [0,T] ,\bbR\rb $, such that 
 \be{eq:FC-uniform} 
 \sup_{(h , \rho )\in \calF_{C, T}} \bbP^{ h, \rho}_T\lb K_\epsilon^c\rb \leq \epsilon .
 \ee
 {Indeed, if $h\equiv 0$ and the initial density $\rho$ satisfies the second bound in 
 \eqref{eq:FC-cond}, then uniform tightness follows directly from Section~8 of \cite{Bil99}, Cauchy-Schwarz 
 and translation invariance of Brownian motion. The general case of $h$ satisfying the first bound 
 in \eqref{eq:FC-cond}
 is then incorporated using H\"{o}lder's inequality.}
 
 If $\rho_0$ is a nice initial profile, then, by Theorem~\ref{thm:L-MV}, the family 
 $\lbr \bbP^{h^x , \rho_0 (x, \cdot )}\rbr_{x\in\bbT^d}$ is a subset of $\calF_{C_T , T}$ for every 
 $T<\infty$. Since $\bbT^d$ is compact, 
 we can proceed as in the proof of exponential tightness for Sanov's theorem on Polish 
 spaces in \cite{DZ10}. 
 
 Once exponential tightness is established, it remains to derive weak large deviation upper bounds. 
 Let $\rho_0$ be a nice initial profile and let $\rho$ be the classical solution to the 
 local McKean-Vlasov equation 
 \eqref{eq:L-MF-MV}.  We have to check that, for any $\epsilon >0$, we can find 
 $\chi_T (\epsilon )>0$, such
 that the 
 following holds:
 
 Let $\mu = \mu [0,T ]\in \sfC \lb [0,T] , \bbM_1 (\bbR\times\Tod )\rb$ is such that  (recall \eqref{eq:LP-met})
 \be{eq:D-dist-eps}
  \sfD_T \lb \mu , \rho_\cdot  (x, \theta ) {\rm e}^{-{2}\psi (\theta )}
  \dd x\dd\theta \rb  
  := \max_{t\in [0, T]}\dd_{\mathsf{LP}}
\lb \mu_t , \rho_t (x, \theta ) 
{\rm e}^{-{2}\psi (\theta )}\dd x\dd\theta \rb \geq 3\epsilon .
 \ee
Then, 
\be{eq:weak-ub} 
\limsup_{\delta\to 0}\limsup_{N\to\infty}
\frac{1}{N^d} \log \hat\bbP_T^N \lb  
\sfD_T \lb \mu ,\mu^N  \rb \leq \delta \rb \leq -\chi_T (\epsilon ). 
\ee
Indeed, in  light  of all the information which we have already collected, 
\eqref{eq:weak-ub} is just a simple concentration upper bound. If 
$\sfD_T \lb \mu , \rho_\cdot  (x, \theta ) {\rm e}^{-{2}\psi (\theta )}
  \dd x\dd\theta \rb \geq 3\epsilon$, then there exists $t\in [0,T]$ 
  and $f\in \mathsf{Lip}\lb \Tod\times \bbR \rb$, such that 
\be{plop.13}
\left| \int f (x , \theta )\mu_t \lb \dd x,\dd\theta \rb -
\int f (x , \theta ) 
\rho_t  (x, \theta ) {\rm e}^{-{2}\psi (\theta )}   \dd x,\dd\theta  
\right| \geq 2\epsilon .
\ee
Since $\rho$ is continuous (as a strong solution to \eqref{eq:L-MF-MV}) in $x$,
we conclude that for $\delta < \epsilon$ and $N$ sufficiently large, the event 
$\lbr \sfD_T \lb \mu ,\mu^N  \rb \leq \delta\rbr$ is included in 
\be{eq:f-event} 
\lbr \left| \frac{1}{N^d}\sum_{i\in \TodN} 
\lb f \lb \frac{i}{N}, \theta_i (t )\rb - \hat\bbE^N_T 
f \lb \frac{i}{N}, \theta_i (t )\rb\rb \right| > \epsilon .
\rbr 
\ee
As in the case of \eqref{eq:uni-SG}, under ${\hat \bbP_T^N}$ the family of
centred random variables 
\be{eq-xi-fam}
\lbr f \lb \frac{i}{N}, \theta_i (t )\rb - \hat\bbE^N_T 
f \lb \frac{i}{N}, \theta_i (t )\rb 
\rbr_{f\in \in \mathsf{Lip}\lb \Tod\times \bbR \rb , t\in [0,T] , N\in \bbN , 
i\in\TodN} 
\ee
 is 
uniformly sub-Gaussian, and \eqref{eq:weak-ub}  follows by the exponential 
Chebyshev inequality. 

\section{Propagation of Chaos.} 
\label{sec:PC}
For the remaining two sections we shall fix 
a nice initial profile $\rho_0$ and 
assume that  the initial density $f^N_0$ 
is in the product form, that is 
\be{eq:D}
\text{
$f_0^N = \rho^N_0$, where $\rho^N_0$ is given by 
\eqref{eq:rho-N} } 
\tag{D}
\ee
Given $k$ distinct points 
$x_1, \dots , x_k\in \Tod$, let $\calP^N_{T ; x_1 , \dots , x_k}$ be the $\bbP_T^N$-marginal 
distribution  on $\sfC \lb [0,T] , \bbR^k \rb$  of $k$ coordinates 
$\lb \theta_{i_1} [0,T] , \dots , \theta_{i_k}[0,T]\rb$, where, for $\ell =1, \dots , k$ 
we set 
$i_\ell = \lfloor N x_\ell \rfloor$.  

Consider the (unique) classical solution $\lb \rho , h\rb$ to \eqref{eq:L-MF-MV}, and let 
$\hat\theta_{x_1}, \dots , \hat\theta_{x_k}$ be independent diffusions, 
\be{eq:theta-l} 
\dd \hat\theta_{i_\ell} (t ) =  \lb h^{x_\ell}(t) - \psi^{\prime} \lb \hat
\theta_{{i_\ell}} (t )\rb\rb \dd t +\dd B_{i_\ell} (t )
\ee
with initial densities $\rho_0 (x_1 , \theta ){\rm e}^{-{2}\psi (\theta )}, \dots , 
\rho_0 (x_k , \theta ){\rm e}^{-{2}\psi (\theta )}$. 
We use $\hat \calP^N_{T ; x_1 , \dots , x_k}$ for  their product distribution 
on $\sfC \lb [0,T] , \bbR^k \rb$.  
\begin{theorem} 
\label{thm:chaos} 
For any nice initial profile $\rho_0$, for any $k = 1, 2, \dots $ points $x_1, \dots , x_k \in \Tod$ , 
and for any finite $T$, 
\be{eq:chaos} 
\lim_{N\to\infty} \calH\lb \calP^N_{T ; x_1 , \dots , x_k}\big| {\hat\calP}^N_{T ; x_1 , \dots , x_k}\rb = 0. 
\ee
\end{theorem}
\begin{proof} 
 Let $\hat\bbP^N_{T, x_1 , \dots , x_k}$ be the distribution of the coupled family of diffusions 
 $\utheta (t ) = \lbr \theta_i (t)\rbr_{i\in\bbT^d}$ with initial product distribution 
 $\rho^N_0$, such that {the} following statements hold:
 
 \noindent 
 a. If $i = i_l = \lfloor N x_\ell \rfloor$, then $\theta_i$ satisfies SDE \eqref{eq:theta-l}. 
 
 \noindent 
 b. Otherwise, $\theta_i$ satisfies \eqref{eq:SDE-N}. 
 
 By construction, the $\hat\bbP^N_{T, x_1 , \dots , x_k}$-marginal  
 distribution of $\lb \theta_{i_1}, \dots , \theta_{i_k}\rb$ is exactly $\hat \calP^N_{T ; x_1 , \dots , x_k}$. 
 Hence 
 \be{plop.1}
  \calH\lb \calP^N_{T ; x_1 , \dots , x_k}\big| \, \hat \calP^N_{T ; x_1 , \dots , x_k}\rb \leq 
  \calH\lb \bbP^N_{T}\big| \hat \bbP^N_{T ; x_1 , \dots , x_k}\rb .
 \ee
We shall proceed with deriving a vanishing, as $N\to\infty$, 
upper bound on  the latter entropy. By Girsanov's formula, 
\be{eq:Girs-chaos} 
\calH\lb \bbP^N_{T}\big|\,  \hat \bbP^N_{T ; x_1 , \dots , x_k}\rb = 
\frac{1}{2}\sum_{\ell=1}^{k}\int_0^T \bbE^N_T \lb h^{x_\ell}(t) - h^{i_\ell}\lb \utheta_t\rb\rb^2 \dd t, 
\ee
recall the definition \eqref{eq:ent-H} of $h^i (\utheta )$. 

All the above terms have the same form, so it is  enough to consider the case $k=1$. 
Let $x\in \Tod$ and $i =  \lfloor N{x}  \rfloor$. 
For $R>0$, consider the  cutoff, $\varphi_R (\theta )$,
of $\theta$, given by
\be{eq:var-phi-R} 
\varphi_R (\theta ) = \theta\1_{\abs{\theta}\leq R} +R\1_{\theta >R } - R\1_{\theta <{-} R} .
\ee
By \eqref{eq:L-MF-MV} , 
\be{plop.2}
 h^x (t) = \int_{\Tod}\int_{\bbR} J (y-x )\rho_t (y , \theta )\theta {\rm e}^{-{2}\psi (\theta )}\dd\theta\dd y = 
 h^x_R (t) + g^x_R (t ) , 
\ee
where  
\be{hRgR} 
h^x_R (t) := \int_{\Tod}\int_{\bbR} J (y-x )\rho_t (y , \theta )\varphi_R (\theta )
{\rm e}^{-{2}\psi (\theta )}\dd\theta\dd y .
\ee
Similarly, 
\be{eq:hRgR-theta} 
\begin{split}
h^{{i}}\lb \utheta_t\rb &= \int_{\Tod}\int_{\bbR} J \lb \frac{i}{N} -y \rb \varphi_R (\theta )
\mu^N_t (\dd y, \dd \theta ) + 
\frac{1}{N^d}\sum_{j\in\TodN} J \lb \frac{j-i}{N} \rb \lb \theta_j (t ) - 
\varphi_R \lb  \theta_j (t )\rb\rb \\ 
& := h^{{i}}_R\lb \utheta_t\rb +  g^{{i}}_R\lb \utheta_t\rb .
\end{split}
\ee
The function 
\be{plop.3}
 f^x_R  (y, \theta ) = \frac{1}{R\|J\|_\infty \lb \|\nabla J\|_\infty\vee 1\rb} J (x - y )\varphi_R (\theta ) 
\ee
belongs to $\mathsf{Lip}\lb \Tod\times \bbR \rb$. Hence, Theorem~\ref{thm:HDL} implies that 
\be{eq:lim-1} 
\lim_{N\to\infty} \frac{1}{2}\int_0^T \bbE^N_T \lb h^{x}_R(t) - h^{i}_R\lb \utheta_t\rb\rb^2 \dd t = 0. 
\ee
{
In turn, 
\be{eq:g-1} 
\lb g^{x}_R (t)\rb^2 \leq \max_{y} J^2 (y ) \int_{\Tod}\int_\bbR 
\1_{\abs{\theta}>R} \theta^2 \rho_t (y , \theta ) {\rm e}^{-2\psi_0 (\theta ) }\dd\theta\dd y , 
\ee
and, consequently, by \eqref{eq:L2-rhoht} and the uniform boundedness of $h$, 
\be{eq:lim-2} 
\lim_{R\to\infty} \int_0^T  \lb g^{x}_R (t)\rb^2 \dd t = 0.  
\ee
}
Finally,  
\be{eq:lim-3}
\bbE^N_T\lb g^{{i}}_R\lb \utheta_t\rb\rb^2 \leq 
\frac{\|J\|_\infty^2}{N^d}
\bbE^N_T \lb 
\sum_{j\in\TodN} 
\lb \theta_j (t ) - \varphi_R\lb \theta_j (t)\rb\rb^2 
\rb.
\ee 
By the entropy inequality, 
\be{eq:lim-3-1}\begin{split}
 \bbE^N_T \lb 
\frac{1}{N^d}\sum_{j\in\TodN} 
\lb \theta_j (t ) - \varphi_R\lb \theta_j (t)\rb\rb^2 
\rb &\leq \frac{1}{N^d}\calH\lb \bbE^N_T \, \big|\,  \hat\bbE^N_T\rb \\ 
&+  
\frac{1}{N^d}\sum_{j\in\TodN} \log\lb \int_\bbR \rho_t\lb \frac{j}{N} , \theta\rb {\rm e}^{-{2}\psi (\theta ) + 
(\theta - \varphi_R (\theta ))^2 }\rb \dd\theta  . \end{split}
\ee
By Proposition~\ref{thm:A},  the first term on the right hand side of \eqref{eq:lim-3-1} tends
to zero as $N\to\infty$. On the other hand, if a density $\rho$ is such that $\|\rho\|_{2, \psi} <\infty$,
then 
\be{plop.4}
 \int_\bbR \rho (\theta ) {\rm e}^{-{2}\psi (\theta ) + (\theta - \varphi_R (\theta )^2}\dd \theta 
 \leq 1 + \|\rho\|_{2, \psi}\sqrt{\int_{\abs{\theta}>R} {\rm e}^{-{2}\psi (\theta ) + 2\theta^2}\dd\theta} . 
\ee
By \eqref{eq:max-L2}, the norms $\| \rho_t (x , \cdot )\|_{2, \psi}$ are uniformly 
bounded in $x\in \Tod$ and 
$t\in [0,T ]$. By our assumption on $\psi$ , 
\be{plop.5}
 \lim_{R\to\infty} \int_{\abs{\theta}>R} {\rm e}^{-{2}\psi (\theta ) + 2\theta^2}\dd\theta = 0, 
\ee
Hence, 
\be{plop.6}
 \lim_{R\to\infty}\lim_{N\to\infty} \bbE^N_T\lb g^{i_\ell}_R\lb \utheta_t\rb\rb^2 = 0, 
\ee
which concludes the proof. 
\end{proof}

\section{Large deviations} 
\label{sec:LD}
Large deviations for a rather  general class of locally mean-field type
models were investigated in \cite{Patrick} via 
a careful adaptation of ideas and
techniques which were originally introduced by Dawson and G{{\"a}}rtner 
\cite{DawGarLDP87, DawGartAMSc89}.  It seems, however, that in a 
particular case we consider here, our results on the existence and uniqueness of 
strong solutions to the system \eqref{eq:L-MF-MV} and, accordingly, 
on hydrodynamic limits towards these strong solutions, pave the way to for a
simpler and more transparent proof of the large deviation principle for 
the law $\bbP^N_T$ of the empirical measure $\mu^N = \mu^N [0,T]$ on 
 $\sfC \lb [0,T] , \bbM_1 (\bbR\times\Tod )\rb$, which relies on martingale 
 techniques of \cite{Com87, KO90}, see also Section~4.2.1 of \cite{GA04}
 for a very clear exposition of the method. Below we sketch the 
 corresponding argument. As, however, explained in the 
 concluding Subsection~\ref{sub:LD-strat} there is 
 an approximation issue still to be settled. 

 \subsection{Exponential tightnes} 
 \label{sub:exp-t} 
 Recall that exponential tightness for the decoupled family $\hat\bbP_T^N$ was 
 already established in Subsection~\ref{sub:thmconc-PN-hat}. Following the notation 
 introduced in Subsection~\ref{sub:thmA} define
 \[ 
  \calN_t^N =\sum_i \int_0^t \lb h_i (\utheta_s ) - h^i_s\rb \dd B_i (s ) .
 \]
Then, ${\rm e}^{q\calN_t^N -\frac{q^2}{2}\la \calN_t^N\ra} $ is a 
$\hat \bbP^N$-martingale for any $q\in \bbR$. Let $A\subset 
\sfC \lb [0,T] , \bbM_1 (\bbR\times\Tod )\rb$ be a measurable 
subset.  
Pick positive $q,p$ and $r$ such that $\frac{1}{q} +\frac{1}{p}+ 
\frac{1}{r} = 1$. By Girsanov's formula, 
and then by H\"{o}lder's  inequality, 
\be{eq:PN-tight-bound} 
\begin{split}
\bbP^N_T (A) 
& = \hat\bbE_T^N\lb \1_A {\rm e}^{ \calN_T^N -\frac{1}{2}\la \calN_T^N\ra}\rb  
 = \hat\bbE_T^N\lb \1_A {\rm e}^{ \calN_T^N -\frac{q}{2}\la \calN_T^N\ra} 
{\rm e}^{\frac{q-1}{2}\la \calN_T^N\ra} \rb \\
&\leq \sqrt[p]{\hat\bbP^N_T (A)} 
\sqrt[q]{\hat\bbE_T^N {\rm e}^{q\calN_T^N -\frac{q^2}{2}\la \calN_T^N\ra}} 
\sqrt[r]{ \hat\bbE_T^N {\rm e}^{\frac{r(q-1)}{2}\la \calN_T^N\ra}} 
= \sqrt[p]{\hat\bbP^N_T (A)} 
\sqrt[r]{ \hat\bbE_T^N {\rm e}^{\frac{r(q-1)}{2}\la \calN_T^N\ra}}
\end{split}
\ee
So, if $K_\gamma\subset \sfC \lb [0,T] , \bbM_1 (\bbR\times\Tod )\rb$ is a compact 
subset satisfying $\hat\bbP_T^N \lb K_\gamma^{\sfc}\rb \leq {\rm e}^{-N^d\gamma}$, then 
\be{eq:PN-gamma} 
\bbP_T^N \lb K_\gamma^{\sfc} \rb 
\leq 
{\rm e}^{-\frac{\gamma}{p}N^d}
\sqrt[r]{ \hat\bbE_T^N {\rm e}^{\frac{r(q-1)}{2}\la \calN_T^N\ra}} .
\ee
Therefore, it remains to check that there exist $\delta >0$ and $C<\infty$ such that
\be{eq:delta-C-bound} 
\hat\bbE_T^N {\rm e}^{\delta\la \calN_T^N\ra} \leq {\rm e}^{C N^d} .
\ee
This follows from (a much stronger statement of) Lemma~\ref{lem:small-delta}. 

We have proved: 
\begin{lemma} 
 \label{lem:exp-tight} 
 If $\rho_0$ is  a nice initial profile in the sence of 
 Definition~\ref{mkprob.2}, then the law $\bbP_T^N$ of 
 $\mu^N$ on $\sfC \lb [0,T] , \bbM_1 (\bbR\times\Tod )\rb$ is exponentially 
 tight for any $T \geq 0$. 
\end{lemma}
\subsection{The rate function and the result} 
\label{sub:rate-f} 
In order to write down an expression for the LD rate function we need
to introduce some additional notation. Let 
$\sfC^{2,0,1}_b\lb \bbR\times\bbT^d\times\bbR\rb$ be the family of bounded 
and continuous 
(with corresponding derivatives) functions $\lb t, x, \theta\rb \mapsto f_t (x , \theta )$. 
Let $\lb\cdot , \cdot\rb_\psi $ be the scalar product of, depending on the context, 
either $\bbL_2 \lb \bbR , {\rm e}^{-{2}\psi (\theta )}\dd\theta \rb$ or 
$\bbL_2 \lb \bbT^d\times\bbR , {\rm e}^{-{2}\psi (\theta )}\dd x\dd\theta \rb$.
 
 For 
 $r_t (x , \theta ) {\rm e}^{-{2}\psi (\theta )} \dd x\dd\theta \in 
 \sfC \lb [0,T] , \bbM_1 (\bbR\times\Tod )\rb$ and 
 $f\in \sfC^{2,0,1}_b\lb \bbR\times\bbT^d\times\bbR\rb$ consider 
 \be{eq:L-form} 
 \calL_T \lb r\, |\,  f \rb := \lb r_T , f_T\rb_\psi - 
 \lb r_0 , f_0\rb_\psi - \int_0^T 
  \lb r_t , \lb \partial_t + L_{h^r (t, x)} \rb f_t \rb_\psi \dd t . 
 \ee
 Above, 
 \be{eq:L-h-x} 
 L_{h (t, x, \theta )} = \frac{1}{2}{\rm e}^{\psi}\partial_\theta\lb  {\rm e}^{-{2}\psi}
 \partial_\theta\rb + h(t, x, \theta )\partial_\theta , 
 \ee
and  
\be{eq:h-r} 
h^r (t, x)  = \int_{\bbT^d} \int_0^\infty J (y -x) \eta r_t (y , \eta ) 
{\rm e}^{-{2}\psi (\eta)}\dd\eta \dd y .
\ee
For any 
$f\in \sfC^{2,0,1}_b$ 
we can extend $\calL_T$ by continuity 
 to  measures $R$ 
which do not have densities with respect to $\dd x\dd\theta$. 
In this way $R \mapsto \calL_T ( R\, |\, f )$  is viewed as a continuous {non-linear} functional on all of
$\sfC \lb [0,T] , \bbM_1 (\bbR\times\Tod )\rb$.

{
For $R\in \sfC \lb [0,T] , \bbM_1 (\bbR\times\Tod )\rb$ define the functional 
$A_T$ via: 
 \be{eq:A-T}
 A_T (R )  = 
 \sup_{f\in \sfC^{2,0,1}_b}\lbr \calL_T ( R \, |\,  f ) - \frac{1}{2}\int_0^T\lb \int_\bbR\int_{\bbT^d} 
 (\partial_\theta f_t )^2 (x ,\theta ) R_t  (\dd x , \dd \theta )\rb \dd t\rbr . 
 \ee 
 }
 Since for any $f\in\sfC^{2,0,1}_b$ the map  
 $R\to \calL_T (R \, |\, f)$ is  {continuous},  the functional 
 $A_T$ is   {lower-semicontinuous}. 
 If $R\neq r_t (x, \theta ){\rm e}^{-{2}\psi (\theta )}\dd\theta \dd x$, then it is easy to check 
 that $A_T (R )= \infty$.  Otherwise, if $r$ is a density of $R$, we shall write 
 $A_T ( r )$ instead of $A_T (R )$.

 As in \cite{KO90} one concludes that if $A_T (r ) <\infty$ then there exists
 a drift field  
 \[ 
b_t (x , \theta ) \in \bbL_{2}\lb \bbR\times\bbT^d\times\bbR , 
r_t (x , \theta ){\rm e}^{-{2}\psi (\theta )}
\dd t\dd x\dd\theta \rb , 
\] 
such that 
\be{eq:L-b} 
\calL_T (r \, |\, f) = \int_0^T \lb  r_t , b_t  \partial_\theta f_t \rb_\psi \dd t,
\ee
which means that 
\be{eq:AT-b} 
A_T (r ) = \frac{1}{2} \int_0^T \lb r_t , b_t^2 \rb_\psi \dd t,\ 
\text{whenever $A_T (r )<\infty$} .
\ee
\begin{remark} 
\label{rem:b-g}
{Since we are working with one dimensional spins, we can always represent $b_t = \partial_\theta g_t$, which defines $g$ up to 
an addition of  $\theta$-independent functions of $(t, x)$.  }
\end{remark}
\begin{theorem} 
 \label{thm:LD} 
 Assume that $\rho_0$ is  a nice initial profile {and assume that the initial density $f_0^N$ satisfies \eqref{eq:D}}. Then the law 
 $\bbP_T^N$ of $\mu^N$ on the space 
 $\sfC \lb [0,T] , \bbM_1 (\bbR\times\Tod )\rb$ 
 satisfies a large deviation principle with rate $N^d$ and with rate function 
 \be{eq:rate-mu-N} 
 I_T (R ) = 
 \begin{cases} 
 &\infty,\ \text{if $R\neq r_t (x, \theta ){\rm e}^{-{2}\psi (\theta )}\dd\theta \dd x$} \\
 & A_T (r ) + \calH\lb r_0\,  \big|\, \rho_0 \rb, \ \text{otherwise}. 
 \end{cases}
 \ee 
\end{theorem}

\subsection{Local mean-field systems with spatially dependent drifts} 
If $A_T (r ) <\infty$, then \eqref{eq:L-b} reads as follows: 
For any $f\in \sfC^{2,0,1}_b\lb \bbR\times\bbT^d\times\bbR\rb$, 
\be{eq:h-b-syst} 
\lb r_T , f_T\rb_\psi - 
 \lb r_0 , f_0\rb_\psi - \int_0^T 
  \lb r_t , \lb \partial_t + L_{h^r (t, x)} + b_t (x, \theta )\partial_\theta 
  \rb f_t \rb_\psi \dd t = 0.
\ee 
Recall \eqref{eq:h-r} how the drift $h^r$ is related to $r$.
In this way, 
\eqref{eq:h-b-syst} is a weak form of a consistent local mean-field 
family of Fokker-Plank equations, and  
the couple $\lb r , h^r\rb$ can be interpreted as a
weak solution to a local mean-field McKean-Vlasov system with an additional
spatially dependent drift $b$.  

Here is the corresponding strong formulation along the lines of 
Definition~\ref{mkprob.1}: 
\begin {definition}\label{def:SLG}
 Given a smooth space-time drift field 
 $b_t (x , \theta )$ and an initial density $\rho_0 (x , \theta )$ a strong 
 solution of the McKean-Vlasov equation is a pair 
 $(\rho,h)$, with $\rho \in \BB$ and $h\in \AA$, 
such that 
 \be{eq:GL-MF-MV}
 \forall\ x\in\Tod\quad \text{and}\ \forall\ t\in 
 [0,\infty) 
 \  
 \begin{cases} 
 &\rho_t (x , \cdot ) = \rho_t^{h^x + b^x}\ 
 \text{with initial condition $\rho_0 (x , \cdot )$} \\
 &\quad {\rm and} \\
 &h^x(t)  = \int\int J(y- x) \rho_t (y , \eta )\eta \eee^{-{2}\psi (\eta )}
 \dd\eta\dd y 
 \end{cases}
 \tag{GL-MV} 
 \ee
 Above $b^x (t, \theta ) = b_t (x, \theta )$ and for any smooth field 
 $u (t , \theta )$ the symbol $\rho_t^u$ stands for the density (under a tacit assumption 
 that it is well defined) with respect
 to ${\rm e}^{-{2}\psi}$ of the one-dimensional diffusion 
 \be{eq:dif-g}
 \dd\theta (t ) = \lb -\psi^\prime (\theta (t)) + u (t, \theta ( t))\rb \dd t
 + \dd B (t ) .
 \ee
 \end{definition}
 
The proofs of  Theorem~\ref{thm:L-MV} and Theorem~\ref{thm:HDL} 
were based on the  a priori bound \eqref{eq:f-bound}. 
The following generalisations of these theorems 
are more or less straightforward: 

 \begin{theorem} 
 \label{thm:GL-MV-HDL}
 Assume that  for all $x\in \bbT^d$ the field $b_t (x , \theta ) := \partial_\theta g_t (x , \theta )$ is smooth 
 in $(t, \theta )$ . Furthermore, assume that  for every $t >0 $
 \be{eq:f-bound-new} 
 \sup_{x\in\bbT^d} \max_{B , s\leq t} 
 \lbr \frac{1}{2} \psi^{\prime\prime} (B) - \partial_s g_s (x , B ) - \lb \psi^\prime (B ) 
  -b_s (x, B )\rb^2 \rbr <\infty 
 \ee
 Let 
 $\rho_0 (x ,\cdot)$ be  a nice initial profile. 
 Then, for any $T >0$,  there exists a 
 { unique strong} 
 solution $(\rho,h)$
 of the 
 system \eqref{eq:GL-MF-MV}. \\
 Furthermore, consider the modified system of coupled diffusions
 \be{eq:SDE-N-b}
\dd \theta^N_i (t) = -\lb \psi^\prime \lb \theta^N_i (t )\rb + b_t\lb \frac{i}{N} , 
\theta^N_i (t)\rb + 
\frac{1}{N^d}\sum_{j\in \TodN}
J\lb\frac{j-i}{N}\rb\theta^N_j (t )\rb \dd t +\dd B_i (t)  ,  \quad i\in \TodN,
\ee
 and let the empirical measure $\mu^N_t$ be defined by \eqref{empdef.1}.
  Let  $\bbP_T^{N, b}$ denote the distribution of $\mu^N [0,T]$
  under the dynamics \eqref{eq:SDE-N-b}. 
 Then, 
 under Assumption~\eqref{eq:D} on the product structure of the initial distribution $f^N_0$,  
  $\bbP_T^{N, b}$ 
 converges 
 to $\delta_{\rho_\cdot (x , \theta ){\rm e}^{-{2}\psi (\theta )}\dd x\dd\theta}$ in the 
 following sense: For any $\epsilon >0$ and $T<\infty$, 
 \be{eq:HDL-b} 
 \lim_{N\to\infty} 
 \bbP_T^{N,b}\lb \max_{t\in [0, T]}\dd_{\mathsf{LP}}
\lb \mu^N_t , \rho_t (x, \theta ) {\rm e}^{-{2}\psi (\theta )}\dd x\dd\theta \rb \geq \epsilon \rb = 0.
 \ee
 \end{theorem}

 \subsection{Scheme of the proof of the LDP Theorem~\ref{thm:LD} }
 \label{sub:LD-strat}
To simplify notation let us write
\be{eq:eps-sim} 
\lbr \mu^N\stackrel{ \epsilon, +}{\sim} R\rbr = \lbr \max_{t\in [0, T]}\dd_{\mathsf{LP}}
\lb \mu^N_t , R_t (x, \dd\theta )
\dd x\rb \leq \epsilon \rbr
\ee
in case of upper bounds, and
\[
\lbr \mu^N\stackrel{\epsilon , -}{\sim} R\rbr = \lbr \max_{t\in [0, T]}\dd_{\mathsf{LP}}
\lb \mu^N_t , R_t (x, \dd\theta )
\dd x\rb < \epsilon \rb, 
\]
in the case of lower bounds. 

In view of exponential tightness we need to derive asymptotic upper and lower 
bounds on $\bbP_T^N\lb \mu^N\stackrel{\epsilon, \pm}{\sim} R\rb$ 
for  any $R\in \sfC \lb [0,T] , \bbM_1 (\bbR\times\Tod )\rb$ and for small 
$\epsilon >0$. 

Proceeding as in \cite{KO90, GA04}  the upper bound with $A_T$ defined in \eqref{eq:rate-mu-N} 
follows 
by Girsanov's theorem: For any fixed $f\in\sfC^{2,0,1}_b\lb \bbR\times\bbT^d\times\bbR\rb$, 
\be{eq:upper-f}
 \bbP_T^N\lb \mu^N\stackrel{\epsilon , +}{\sim} R\rb \leq 
 {\rm e}^
{ -N^d 
 \lb
 \calL_T (R \, |\, f ) - 
 \frac{1}{2}\int_0^T \int_\bbR\int_{\bbT^d} 
 \lb 
 \partial_\theta f_t\rb^2 
 R_t (x, \dd \theta)\dd x\dd t  \rb 
 (1 + {\rm O}_f (\epsilon ))
  }\;
  \bbQ^{N , f}_T\lb \mu^N\stackrel{\epsilon , +}{\sim} R\rb .
\ee
Above, ${\rm O}_f (\epsilon )$ is a quantity which tends to zero as $\epsilon\to 0$, 
and 
\[
 \frac{\dd \bbQ^{N , f}_T}{\dd \bbP^{N }_T} : = 
 {\rm e}^{N^d \calM_T^{N, f} - \frac{N^{2d}}{2} \la \calM_T^{N, f}\ra} , 
\]
where $\calM_t^{N, f}$ is a $\bbP^N$-martingale, 
\[
 \calM_t^{N, f} = \mu^N_t \lb f_t \rb -\mu^N_0 \lb f_0\rb - 
 \frac{1}{N}\sum_i \int_0^t \lb \partial_s + L_{h^i (\utheta (s))}\rb 
 f_s \lb \frac{i}{N} , \theta_i (s )\rb \dd s .
\]
Recall \eqref{eq:ent-H} and \eqref{eq:L-h} to follow the above notation. 

Optimisation over $f$ in the first term on the right hand side of 
\eqref{eq:upper-f} gives $A_T (R )$. On the other hand, 
\be{eq:zero-LD} 
 \bbQ^{N , f}_T\lb \mu^N\stackrel{\epsilon , +}{\sim} R\rb  \leq 
\bbQ^{N , f}_T\lb \dd_{\mathsf{LP}} \lb \mu^N_0 , R_0\rb \leq \epsilon \rb = 
\bbP^{N}_T\lb \dd_{\mathsf{LP}} \lb \mu^N_0 ,  R_0\rb \leq \epsilon \rb. 
\ee
The last expression is subject to stationary Sanov-type large deviations 
with rate function $\calH\lb R_0\, \big|\, \rho_0\rb$. 
\smallskip 

Let us turn to the lower bound:  
{By general methods it will follow from upper bounds if one is able to prove that there is 
always a unique weak solution to \eqref{eq:h-b-syst}.   Below we sketch an alternative route which is based
on the approach to existence and  uniqueness  of strong solutions  and subsequent derivation of
hydrodynamic limits, as developed in  Section~\ref{sec:MF}--\ref{sec:HDL}, and
formulated in Theorem~\ref{thm:GL-MV-HDL}. }
\smallskip 

\noindent 
\step{0} 
Recall notation \eqref{eq:h-r}. In view of the exponential tightness of the random variables $h^{\mu^N} (t, x)$ by the lower semicontinuity 
of the functional $\rho \mapsto h^{\rho} (t, x)$ with respect to the distance $D_T$, we may assume that $\sup_x \max_{t\leq T} \abs{h^{r} (t, x )} $ is bounded. 
\smallskip

\noindent 
\step{1} 
Assume that 
$R_t = r_t (x, \theta ){\rm e}^{-{2}\psi (\theta )\dd\theta \dd x}$ and 
$(r , h )$ is the unique strong solution to \eqref{eq:GL-MF-MV} 
with a
nice initial profile $r_0$ and 
smooth drift field $b = \partial_\theta g$ 
satisfying \eqref{eq:f-bound-new}. 
As in \eqref{eq:upper-f}: 
\be{eq:lower-f}
 \bbP_T^N\lb \mu^N\stackrel{\epsilon , -}{\sim} R\rb \geq 
 {\rm e}^
{ -N^d 
 \lb
 \calL_T (R \, |\, g ) - 
 \frac{1}{2}\int_0^T \int_\bbR\int_{\bbT^d} 
 \lb 
 \partial_\theta g_t\rb^2 
 R_t (x, \dd \theta)\dd x\dd t  \rb 
 (1 + {\rm O} (\epsilon ))
 }\;
  \bbQ^{N , g}_T\lb \mu^N\stackrel{\epsilon ,-}{\sim} R\rb .
\ee
By our choice of $r$,  the expression $ \calL_T (R \, |\, g ) - 
 \frac{1}{2}\int_0^T \int_\bbR\int_{\bbT^d} 
 \lb 
 \partial_\theta g_t\rb^2 
 R_t (x, \dd \theta)\dd x\dd t $ equals to $A_T (r )$. Furthermore,  with the notation introduced in the
  course of the formulation 
 of Theorem~\ref{thm:GL-MV-HDL},     $ \bbQ_T^{N, g } = \bbP_T^{N, b}$.  Finally,
  assuming that $\calH (r_0 |\rho_0 ) <\infty$, 
 \be{eq:PN-comp} 
 \bbP^{N , b}_T\lb \mu^N\stackrel{\epsilon ,-}{\sim} R\rb = \bbP^{N, b}\lb \prod_i \frac{r_0 (i/N , \theta_i (0))}{\rho_0 (i/N , \theta_i (0))}{\rm e}^{-N^d \frac{1}{N^d}
 \sum_i \log \frac{r_0 (i/N , \theta_i (0))}{\rho_0 (i/N , \theta_i (0))}}  ;  \mu^N\stackrel{\epsilon , -}{\sim} R\rb
 \ee
 The measure $\bbP^{N, b}\lb \prod_i \frac{r_0 (i/N , \theta_i (0))}{\rho_0 (i/N , \theta_i (0))}, \cdot\rb $ is just the distribution of 
 $\mu^N [0,T]$ under the sde \eqref{eq:SDE-N-b} and the product initial distribution sampled from the nice initial profile $r_0$. Under this measure, 
the law of large numbers implies that $\lim_{N\to\infty} \frac{1}{N^d}
 \sum_i \log \frac{r_0 (i/N , \theta_i (0))}{\rho_0 (i/N , \theta_i (0))} = \calH (r_0 |\rho_0 )$. On the other hand, 
 \be{eq:lim-prob-rnot} 
 \lim_{N\to\infty} 
 \bbP^{N, b}_T\lb \prod_i \frac{r_0 (i/N , \theta_i (0))}{\rho_0 (i/N , \theta_i (0))} ;  \mu^N\stackrel{\epsilon ,-}{\sim} R\rb = 1, 
 \ee
 by \eqref{eq:HDL-b}. 
 \smallskip 
 
 It remains to show that strong  solutions $r$ described in S{\small TEP}\,1 are dense on the graph of $A_T$. 
 That is for any $r$ with $A_T (r ) <\infty$ there
 exists a sequence  $\lb r^\epsilon, b^\epsilon \rb $ such that for 
 any $\epsilon$  $\lb r^\epsilon, h^{r^\epsilon}\rb$ is a strong 
 solution to \eqref{eq:GL-MF-MV} (with smooth field $b^\epsilon$ satisfying    \eqref{eq:f-bound-new}), and 
 both $\lim r^\epsilon = r $ and $\lim A_T (r^\epsilon ) = 
 A_T (r )$. 
 \smallskip 
 
 \noindent
 \step{2} Let us go back to the linear form $\calL_T  (r~|~f )$ in \eqref{eq:L-form}.  It could be rewritten as 
 \be{eq:Lnot-form} 
 \calL_T  (r~|~f ) = \calL_T^0   (r~|~f ) - \int_0^T \lb r_t ,h^r_t \partial_\theta f_t \rb_\psi \dd t. 
 \ee
 Hence $b_t = b_t^0 - h_t^r$, where $b_t^0$ satisfies $\calL_T^0   (r~|~f )  = \int_0^T \lb r_t ,b_t^0 \partial_\theta f_t \rb_\psi \dd t$ for any $f\in \sfC^{2,0,1}_b$.  Accordingly, 
 \be{eq:AT-not} 
 \calA_T (r ) = \frac{1}{2}\int_0^T \lb r_t , \lb b_t^0 - h_t^r \rb^2 \rb_\psi \dd t .
 \ee
 Therefore one has to show that there exists a sequence $\lbr b_t^\epsilon \rbr$ such that  for 
 any $\epsilon >0$ it complies with \eqref{eq:f-bound-new} , and in addition, 
 $r^\epsilon := \rho^{b^\epsilon}$ satisfies: 
 \be{eq:Cond-r-epsilon} 
 \forall \, T\ D_T-\lim_{\epsilon\to 0} r^{\epsilon} = r\quad\text{and}\quad 
 \lim_{\epsilon\to 0} = \int_0^T \lb r_t^\epsilon  , \lb b_t^\epsilon \rb^2\rb_\psi \dd t = \int_0^T \lb r_t , \lb b_t^0 \rb^2 \rb_\psi \dd t
 \ee 
 Note that the densities $r_t^\epsilon (\theta , x )$ are completely decoupled, 
 and the question is essentially about one-dimensional
 parabolic PDE-s in divergence form. We proceed to discuss the latter.
 \smallskip 
 
 \noindent
 \step{3} Let $r_t (\theta )$ be a density with respect to ${\rm e}^{-2\psi (\theta )}\dd \theta$. Given $f\in \sfC^{2, 1}\lb \bbR\times\bbR_+ \rb$ we, with a slight abuse of notation, 
 continue using 
 \be{eq:LT-not}
 \calL_T^0 (r\, |\, f ) = \lb r_T ,f_T\rb_\psi -  \lb r_0 ,f_0\rb_\psi - \int_0^T \lb r_t , \lb \partial_t + L_0\rb f_t\rb_\psi \dd t . 
 \ee
 We have to check the following: If 
 \be{eq:AT-not-1D} 
 \calA_T^0 (r ) := \sup_{f\in \sfC^{2,1}} \lbr  \calL_T^0 (r\, |\, f ) - \frac{1}{2} \int_0^T \lb r_t , \lb\partial_\theta f_t\rb^2\rb_\psi\dd t\rbr <\infty  , 
 \ee
 or alternatively, if $ \calL_T^0 (r\, |\, f ) = \int_0^T \lb r_t , b_t\partial_\theta f\rb_\psi\dd t$ for some $b\in \bbL_2 \lb \bbR\times [0,T], r_t {\rm e}^{-2\psi} \dd \theta \dd t\rb$, 
 then one can find a sequence $\lbr b^\epsilon_t (\theta )\rbr$ which complies with \eqref{eq:f-bound-new}, and 
 and in addition, 
 the unique classical solution $r^\epsilon $ of 
 \be{eq:r-epsilon} 
 \partial_t r^\epsilon_t - L_0 r^\epsilon_t = - {\rm e}^{2\psi }\partial_\theta \lb {\rm e}^{-2\psi }b_t^\epsilon r_t^\epsilon\rb 
 \ee
 with a nice initial profile $r_0^\epsilon$,  which satisfies 
 \be{eq:Cond-r-epsilon-red} 
 \forall \, T\ \lim_{\epsilon\to 0} \max_{t\leq T} 
 \dd_{\mathsf{LP}} \lb  r^{\epsilon},  r\rb = 0 \quad\text{and}\quad 
 \lim_{\epsilon\to 0} = \int_0^T \lb r_t^\epsilon  , \lb b_t^\epsilon \rb^2\rb_\psi 
 \dd t = \int_0^T \lb r_t , \lb b_t^0 \rb^2 \rb_\psi \dd t.
 \ee


\end{document}